\documentclass{amsart}

\usepackage{etex}
\usepackage{amsmath, amssymb}
\usepackage{array}
\usepackage[frame,cmtip,arrow,matrix,line,graph,curve]{xy}
\usepackage{graphpap, color, paralist, pstricks}
\usepackage[mathscr]{eucal}
\usepackage[pdftex]{graphicx}
\usepackage[pdftex,colorlinks,backref=page,citecolor=blue]{hyperref}
\usepackage{pifont}
\usepackage{cancel}
\usepackage{mathtools}
\usepackage{tikz}

\newtheorem{thm}{Theorem}[section]
\newtheorem{prop}[thm]{Proposition}

\newtheorem{lemma}[thm]{Lemma}
\theoremstyle{definition}

\newcommand{\Z}{\mathbb{Z}}

\newcommand{\mA}[0]{\mathfrak{A}}
\newcommand{\mB}[0]{\mathfrak{B}}
\newcommand{\mC}[0]{\mathfrak{C}}
\newcommand{\mD}[0]{\mathfrak{D}}
\newcommand{\mE}[0]{\mathfrak{E}}
\newcommand{\mF}[0]{\mathfrak{F}}
\newcommand{\mG}[0]{\mathfrak{G}}
\newcommand{\mH}[0]{\mathfrak{H}}

\newcommand{\mL}[0]{\mathfrak{L}}

\newcommand{\mQ}[0]{\mathfrak{Q}}
\newcommand{\mR}[0]{\mathfrak{R}}
\newcommand{\mS}[0]{\mathfrak{S}}

\newcommand{\mV}[0]{\mathfrak{V}}
\newcommand{\mX}[0]{\mathfrak{X}}
\newcommand{\mZ}[0]{\mathfrak{Z}}

\newcommand{\pr}{\mathbb{P}}

\newcommand{\E}[0]{\mathbb{E}}

\newcommand{\beq}[1]{\begin{equation}\label{#1}}
\newcommand{\enq}[0]{\end{equation}}

\newcommand{\mn}[0]{\medskip\noindent}
\newcommand{\nin}[0]{\noindent}

\newcommand{\sub}[0]{\subseteq}
\newcommand{\sm}[0]{\setminus}
\renewcommand{\dots}[0]{,\ldots,}

\newcommand{\ov}[0]{\overline}

\newcommand{\less}[0]{~~\mbox{\raisebox{-.6ex}{$\stackrel{\textstyle{<}}{\sim}$}}~~}
\newcommand{\more}[0]{~~\mbox{\raisebox{-.9ex}{$\stackrel{\textstyle{>}}{\sim}$}}~~}

\newcommand{\B}[0]{{\mathcal B}}

\newcommand{\f}[0]{{\mathcal F}}

\newcommand{\g}[0]{{\mathcal G}}
\newcommand{\h}[0]{{\mathcal H}}

\newcommand{\K}[0]{{\mathcal K}}

\newcommand{\T}[0]{{\mathcal T}}

\newcommand{\W}[0]{{\mathcal W}}

\newcommand{\Y}[0]{{\mathcal Y}}
\newcommand{\ZZZ}[0]{{\mathcal Z}}

\newcommand{\bgs}[0]{{\boldsymbol \gs}}

\newcommand{\pY}[0]{p_Y}
\newcommand{\gcY}[0]{\gc_{_Y}}

\newcommand{\bF}[0]{\boldsymbol\f}
\newcommand{\bh}[0]{\boldsymbol h}
\newcommand{\bG}[0]{\boldsymbol\g}
\newcommand{\bH}[0]{\boldsymbol\h}

\newcommand{\bT}[0]{\boldsymbol\T}

\newcommand{\bZ}[0]{\boldsymbol Z}

\newcommand{\ra}[0]{\rightarrow}

\newcommand{\ff}[0]{{\bf f}}

\newcommand{\XX}[0]{\textbf{X}}

\newcommand{\TTT}[0]{T}

\newcommand{\uu}[0]{U}

\newcommand{\ww}{\mbox{{\sf w}}}

\newcommand{\0}[0]{\emptyset}

\newcommand{\C}[2]{{{#1}\choose{{#2}}}}
\newcommand{\Cc}[0]{\tbinom}
\newcommand{\ga}[0]{\alpha }
\newcommand{\gb}[0]{\beta }
\newcommand{\gc}[0]{\gamma }
\newcommand{\gd}[0]{\delta }
\newcommand{\gD}[0]{\Delta }
\newcommand{\gG}[0]{\Gamma }

\newcommand{\gl}[0]{\lambda }
\newcommand{\gL}[0]{\Lambda}
\newcommand{\go}[0]{\omega}
\newcommand{\gO}[0]{\Omega}

\newcommand{\gs}[0]{\sigma}

\newcommand{\gz}[0]{\zeta}
\newcommand{\eps}[0]{\varepsilon }
\newcommand{\vt}[0]{\vartheta}
\newcommand{\vs}[0]{\varsigma}

\newcommand{\vp}[0]{\varphi}

\def\maxr{{\rm maxr \,\,}}
\newcommand{\gdz}[0]{\delta_0 }

\newcommand{\Rastar}[0]{~~\mbox{\raisebox{-.05ex}{$\stackrel{*}{\Longrightarrow}$}}~~}

\newcommand{\comments}[1]{}

\begin{document}

\title{Asymptotics for Shamir's Problem}

\author{Jeff Kahn}
\thanks{Supported by NSF Grant DMS1501962, BSF Grant 2014290,
and a Simons Fellowship.}
\email{jkahn@math.rutgers.edu}
\address{Department of Mathematics, Rutgers University \\
Hill Center for the Mathematical Sciences \\
110 Frelinghuysen Rd.\\
Piscataway, NJ 08854-8019, USA}

\begin{abstract}

For fixed $r\geq 3$ and $n$ divisible by $r$, let 
$\bH=\bH^r_{n,M}$ be the random $M$-edge $r$-graph on
$V=\{1\dots n\}$; that is, $\bH$ is chosen uniformly from the $M$-subsets of $\K:=\C{V}{r}$ ($:= \{\mbox{$r$-subsets of $V$}\}$).
\emph{Shamir's Problem} (circa 1980) asks, roughly,
\begin{center}
\emph{for what $M=M(n)$ is $\bH$ likely to contain a perfect matching}
\end{center}
(that is, $n/r$ disjoint $r$-sets)?

In 2008 Johansson, Vu and the author showed that this is true for $M>C_rn\log n$.  The present paper has two purposes.  First, it establishes the 
asymptotically correct version of the 2008 result:

\mn
\textbf{Theorem 1.}
\emph{For fixed $\eps>0$ and $M> (1+\eps)(n/r)\log n$,} 
\[
\pr(\bH ~\mbox{\emph{contains a perfect matching}})\ra 1 ~ 
\mbox{\emph{as} $n\ra\infty$}.
\]

Second, it begins
a proof of the definitive ``hitting time" statement:

\mn
\textbf{Theorem 2.}
\emph{If $A_1, \ldots ~$ is a uniform permutation of $\K$,
$\bH_t=\{A_1\dots A_t\}$, and
\[
T=\min\{t:A_1\cup \cdots\cup A_t=V\},
\]
then
$\pr(\bH_T ~\mbox{contains a perfect matching})\ra 1 ~
\mbox{\emph{as} $n\ra\infty$}$.}

\mn
It is shown here that Theorem~2 follows from a conditional version of Theorem~1
that will be proved elsewhere.  
The key ideas in that proof are similar to those for Theorem~1, but the argument
is a longer story, and it has seemed best to give the present
separate proof of Theorem~1, in which those ideas may appear more clearly.

\end{abstract}

\maketitle

\section{Introduction}\label{Intro}

A (simple) \emph{r-graph} (or \emph{r-uniform hypergraph}) is
a set $\h$ of $r$-subsets (\emph{edges}) of a
\emph{vertex} set $V=V(\h)$;
a \emph{matching} of such an $\h$ is a set of disjoint edges;
and a \emph{perfect matching} (p.m.)
is a matching of size $|V|/r$.
Write $\bH^r_{n,M}$ for the random $M$-edge $r$-graph on
$[n]:=\{1\dots n\}$; that is, $\bH^r_{n,M}$ is chosen uniformly from the $M$-subsets of $\K:=\C{[n]}{r}$.
(Some notation is collected at the end of this section.)


We are interested here in \emph{Shamir's Problem,} which asks, roughly,
with $n$ ranging over (large) multiples of a fixed $r$,
\begin{center}
\emph{for what $M$ is $\bH^r_{n,M}$ likely to contain a perfect matching?}
\end{center}

\nin
In what follows we will always work with a fixed $r$ and omit this
from our notation---so $\bH_{n,M}$ is $\bH^r_{n,M}$---and restrict to
$n$ divisible by $r$.


Shamir's Problem first appeared in print in
\cite{Erdos}, where Erd\H{o}s says he
heard it from Eli Shamir in 1979,
and,
following initial results in \cite{SS}, became one of the most intensively studied
questions in probabilistic combinatorics; for example, it and its graph factor
analogue (see below) were,
according to \cite[Section 4.3]{JLR},
``two of the most challenging, unsolved problems
in the theory of random structures."


For a more precise question, recall
that $M_0=M_0(n)$ is \emph{a threshold} for the property of containing a perfect matching if
\beq{ERth}
\pr(\mbox{$\bH_{n,M}$ has a perfect matching})\ra \left\{\begin{array}{ll}
0&\mbox{if $M/ M_0\ra 0$,}\\
1&\mbox{if $M/ M_0\ra \infty$.}
\end{array}\right.
\enq
This notion was introduced by Erd\"os and R\'enyi in \cite{ER}
and has been a central concern of probabilistic combinatorics since that
time (see e.g.\ \cite{JLR}).
Note \eqref{ERth} depends only on the order of magnitude of $M_0$,
though ``the threshold" is a common abuse.


A natural guess---maybe with some hindsight; see below---is that, for any
(fixed) $r$,
\beq{threshold}
\mbox{$n\log n$ is a threshold for containing a perfect matching.}
\enq
(When it matters---here it does not---$\log$ is natural logarithm.)
We think of this
as crudely expressing the idea that \emph{in the random setting}
the main obstacle to existence of a perfect matching is isolated vertices
(vertices not in any edges)---which typically disappear when $M\approx (n/r)\log n$.

Note that while \eqref{threshold} seems natural (or obvious) today,
it was not always so.  For example Erd\H{o}s \cite{Erdos} says
``... usually one can guess the answer [for random hypergraph problems]
almost immediately.
Here we have no idea ...," and \cite{SS} gives no guess as to the threshold.
It was only in \cite{CFMR} that \eqref{threshold} (in the stronger form
\eqref{cfmrconj} below) was first suggested in print,
though its likelihood
was surely recognized before then.


Following various attempts, the most successful in
\cite{FJ} and \cite{Kim}
(see also e.g.\ \cite{CFMR,Kriv2}), \eqref{threshold} was proved in \cite{JKV}:
\begin{thm}\label{jkv}
For each $r$ there is a $C_r$ such that
if $M>C_rn\log n$ then $\bH_{n,M}$ has a perfect matching
w.h.p.\footnote{``with high probability," meaning with probability tending to 1 as $n\ra \infty$}
\end{thm}
\nin
(See also \cite[Sec.\ 13.2]{Frieze-Karonski} for an exposition.)


The challenge since Theorem~\ref{jkv}
has been to show that isolated vertices
are more literally the issue.
Ideally this means proving the precise
\emph{hitting time} statement:
if $A_1\ldots ~$ is a uniform permutation of $\K$
then
w.h.p.\ the $A_i$'s include a p.m.\ as soon as they cover the vertices.
(This possibility is suggested in \cite{JKV}, but was by then an obvious guess.)

A somewhat less ambitious goal is to show
that Theorem~\ref{jkv} holds for any (fixed) $C_r>1/r$.
We may call this \emph{asymptotics of the threshold}:
it gives $M_c\sim (n/r)\log n$,
where $M_c=M_c(n)$ is \emph{the}
threshold, the least $M$ for which
\[
\pr(\mbox{$\bH_{n,M}$
has a perfect matching})\geq 1/2
\]
(which \emph{is} a threshold in the Erd\H{o}s-R\'enyi sense; see \cite{BTth}
or \cite[Theorem 1.24]{JLR}). 
It is easy to see that this asymptotic
version does follow from the hitting time statement.

The conjecture of
\cite{CFMR} is stronger than asymptotics of the threshold
but weaker than the hitting time version:
if $M= (n/r)(\log n +c_n)$, then
\beq{cfmrconj}
\pr(\mbox{$\bH_{n,M}$ has a perfect matching})\ra \left\{\begin{array}{ll}
0&\mbox{if $c_n\ra 0$,}\\
e^{-e^{-c}}&\mbox{if $c_n\ra c$,}\\
1&\mbox{if $c_n\ra\infty$.}
\end{array}\right.
\enq
Equivalently, the probability that $\bH_{n,M}$ avoids isolated vertices yet
fails to contain a perfect matching tends to zero.
For $r=2$, \eqref{cfmrconj} and the hitting time statement were shown by
Erd\H{o}s and R\'enyi \cite{ERPM} and
Bollob\'as and Thomason \cite{BT} respectively.
(So Erd\H{o}s' comment above might suggest he believed
the answer for larger $r$ would be different.)

In this paper we show that
the asymptotics of the threshold are as expected and give
the first step
in a proof of the hitting time statement
that will be completed in \cite{HT}; thus:

\begin{thm}\label{ThmX}
For fixed $\eps>0$ and $M> (1+\eps)(n/r)\log n$, $\bH_{n,M}$
has a perfect matching w.h.p.
\end{thm}

\begin{thm}\label{ThmY}
If $A_1, \ldots ~$ is a uniform permutation of $\K$,
$\bH_t=\{A_1\dots A_t\}$, and
\[
T=\min\{t:A_1\cup \cdots\cup A_t=V\}
\]
(the \emph{hitting time}),
then
$\bH_T$ has a perfect matching w.h.p.
\end{thm}

\nin

[We note in passing that Theorem~\ref{ThmX}
is equivalent to its analogue for
$\bH_{n,p}=\bH^r_{n,p}$ (the random $r$-graph on $[n]$ in which each edge is present
with probability $p$, independent of other choices):
\begin{thm} \label{ThmXp}
For fixed $\eps>0$ and $p > (1+\eps)\C{n-1}{r-1}^{-1}\log n$,
$\bH_{n,p}$ has a perfect matching w.h.p.
\end{thm}
\nin
See e.g.\ Propositions 1.12 and 1.13 of \cite{JLR} for the equivalence.]

The story of these results ran
very much contrary to expectations.  The author had felt since \cite{JKV}
that a proof of Theorem~\ref{ThmX}
might not be out of the question (maybe a minority opinion),
but
that Theorem~\ref{ThmY} was probably hopeless;
but in retrospect it is the former that was
the bigger step.

Theorem~\ref{ThmY}
is proved by reducing to a statement like
Theorem~\ref{ThmX},
but in a conditional space where even
routine points from the proof of Theorem~\ref{ThmX} are
not straightforward;
so the present organization, with its separate proof of the now
subsumed Theorem~\ref{ThmX}, is intended to focus
on what seem the most
important points.
(It should also make the proof of Theorem~\ref{ThmY} in
\cite{HT} easier to follow, and will somewhat shorten \cite{HT}, since
some of what we do here \emph{can} be used there directly.)

As in \cite{JKV} our approach to Theorems~\ref{ThmX} and \ref{ThmY}
depends crucially on working with counting
versions;
with $\Phi(\h)$ denoting the number of perfect matchings of $\h$,
the corresponding stronger statements are:
\begin{thm} \label{ThmX'}
For fixed $\eps>0$
and $M> (1+\eps)(n/r)\log n$,
w.h.p.
\beq{MainIneq}
\Phi(\bH_{n,M}) > \left[e^{-(r-1)}rM/n\right]^{n/r}e^{-o(n)} .
\enq
\end{thm}
\begin{thm}\label{ThmY'}
For $\bH_t$ and $T$ as in Theorem~\ref{ThmY}, w.h.p.
\beq{TY'bd}
\Phi(\bH_T) > \left[e^{-(r-1)}\log n\right]^{n/r}e^{-o(n)} .
\enq
\end{thm}
\nin
The right-hand sides 
are (of course) roughly the expectations of the
left-hand sides; more precisely,
they are within subexponential factors of those expectations.

In Section~\ref{Skeleton} we derive Theorem~\ref{ThmX} from
several statements whose proofs will be the main work of this paper.
Outlining that work will be easier once we have
the framework of Section~\ref{Skeleton}, so is postponed until then.

Section~\ref{Reduction} gives the reduction that is the
first step in the proof of Theorem~\ref{ThmY}.
(To be precise, we reduce Theorem~\ref{ThmY'} to Theorem~\ref{ThmZ},
a conditional variant of Theorem~\ref{ThmX'}.
The same reduction gets Theorem~\ref{ThmY} from the weaker version of
Theorem~\ref{ThmZ} corresponding to Theorem~\ref{ThmX}, but, again,
we don't know how to prove the weaker version without proving the stronger.)

\nin
\emph{Graph factors}

Recall that, for graphs $H$ and $G$, an \emph{$H$-factor} of $G$ is a
collection of copies of $H$ in $G$ (subgraphs of $G$ isomorphic to $H$)
whose vertex sets partition $V(G)$.
The graph-factor counterpart of Shamir's Problem asks (roughly),
\emph{when is the random graph $G_{n,M}$ likely to contain an $H$-factor?}
This was originally suggested (for $H=K_3$) by
Ruci\'nski \cite{Rucinski2}.

Here the naive guess---that vertices not in copies of $H$ are the main
obstruction---is not always correct, though one does expect it to
be correct for \emph{strictly balanced} $H$ (see \cite{JLR})
and \emph{slightly} beyond.
For strictly balanced $H$ it is shown in \cite{JKV} at the level
of Erd\H{o}s-R\'enyi thresholds (so the analogue of Theorem~\ref{ThmX};
this says, for example, that $n^{4/3}\log^{1/3}n$ is a threshold for
existence of a triangle-factor).
See Conjecture~1.1 of \cite{JKV} for what ought to
be true in general.  
Though given in detail only for graphs, the arguments and results of
\cite{JKV} extend essentially unmodified to \emph{r-graph}-factors,
where Theorem~\ref{ThmX} is just the case that $H$ consists of a single edge.

I expect---admittedly, without having thought very seriously---that
the present results
extend to the general
graph (and hypergraph) factor setting, with,
as was true in \cite{JKV}, some technical complications but
all key ideas already appearing in the arguments for Shamir.
Beautiful recent coupling arguments of O.\ Riordan \cite{Riordan}
and A.\ Heckel \cite{Heckel} show
that in some cases---e.g.\ cliques---the graph factor versions of
Theorems~\ref{jkv} and \ref{ThmX} \emph{follow from} the Shamir versions
(e.g.\ at these levels of accuracy, Ruci\'nski's triangle-factor
question is \emph{contained in} Shamir);
but there seems
little chance of anything analogous for Theorem~\ref{ThmY}.

\nin
\textbf{Usage}

Throughout the paper we fix $r\geq 3$; take $V=[n]:=\{1\dots n\}$,
with $r|n$;
and use $\K$ for $\C{V}{r}$.
We use $v,w,x,y,z$ for vertices and
$\f,\g, \h$ for $r$-graphs (a.k.a.\ subsets of $\K$),
or bold versions of these when the $r$-graphs in question are random.
As above, we
abbreviate $\bH^r_{n,M}=\bH_{n,M}$,
and the number of perfect matchings (or p.m.s) of $\h$ is denoted $\Phi(\h)$.

We use
$\h_x =\{A\in \h:x\in A\}$;
$\gD_\h$, $\gd_\h$ and $D_\h$ for maximum, minimum and average degrees in $\h$;
and $\h-X=\{A\in \h: A\sub V\sm X\}$.
As usual the \emph{codegree} (in $\h$) of $x,y$ is $|\{A\in \h: x,y\in A\}|$.
We will often abusively write $Y\cup x$ and $Y\sm x$ for $Y\cup \{x\}$ and $Y\sm \{x\}$.

Asymptotic notation is interpreted as $n \ra \infty$
(with dependence on $n$ typically suppressed).
We use $a\ll b$ and $a=o(b)$ interchangeably and, similarly,
$a\less b$ is the same as $a<(1+o(1))b$.
We use both ``a.e." and ``a.a."
to mean ``for all but a $o(1)$-fraction."
A familiar point that nonetheless seems worth mentioning:
given $\eps$, implied quantities in asymptotic expressions
not mentioning $\eps$ (constants in $O(\cdot)$ and $\gO(\cdot)$, rates in $o(\cdot)$ and
$\go(\cdot)$) depend on $\eps$, but, for example, the implied constant in $O(\eps)$
does not.

We use $\log$ for natural logarithm and
$a\pm b$ for a quantity within $b$ of $a$.
We will always assume $n$ is large enough to
support our assertions and, following a common abuse,
usually pretend large numbers are integers.

We will sometimes use bold for random objects: 
consistently for 
$r$-graphs, but otherwise 
only if it seems needed to distinguish a random object from its possible values.
We use mathfrak characters ($\mA,\mB,\mC,\mD,\mE, \ldots$)
for properties (saying, e.g., ``$\h$ has property $\mA$,"
``$\h$ satisfies $\mA$," ``$\h\models\mA$" as convenient) and events
(e.g.\ $\mA_t=\{\bH_t\models \mA\}$; see Section~\ref{BandR}),
and will usually prefer $\mA\mR$ to $\mA\wedge\mR$.

\section{Skeleton}\label{Skeleton}

Here we prove Theorem~\ref{ThmX'} modulo
a few assertions whose justification will be the main content of the paper.
As mentioned earlier, the approach is similar to that of \cite{JKV};
a major, if nearly invisible, difference is the $o(n)$ in \eqref{MainIneq}, which
was formerly $O(n)$;  see ``Orientation" below for a \emph{little} more on this.


Fix $\eps$, let $M$ be as in the statement
of Theorem~\ref{ThmX'} (or \ref{ThmX}),
and set $\TTT=\C{n}{r}-M$.
Let $A_1\dots A_{\C{n}{r}}$ be a uniform ordering of $\K$
($=\C{V}{r}$)
and set $\h_t = \K\sm\{A_1\dots A_t\}$; so $\h_0 =\K$
and we may take $\h_{n,M}=\h_T$.

(Note the $T$ and $\h_t$ here disagree with---are nearly the opposite of---those
in Theorems~\ref{ThmY} and \ref{ThmY'}, reflecting the different $\h_0$'s
($\K$ vs.\ $\0$), but we will not see those theorems again until
Section~\ref{Reduction}, when we are done with the present setup;
nor will there ever be any danger of confusing $\h_t$ and $\h_x$
($=\{A\in \h:x\in A\}$; see Usage).)

\mn

Set $\Phi(\bH_t) = \Phi_t$ and let $\xi_t$
be the fraction of perfect matchings of $\bH_{t-1}$
that contain $A_t$ (so $\xi_t =\Phi(\bH_{t-1}-A_t)/\Phi_{t-1}$).
Then
\[
\Phi_t = \Phi_0(1-\xi_1) \cdots (1-\xi_t);
\]
equivalently,
\beq{mgt1}
\mbox{$\log \Phi_t = \log \Phi_0 +\sum_{i=1}^t\log (1-\xi_i).$}
\enq

It will be helpful to set
\beq{gL}
\gL = (r-1)n/r;
\enq
this quantity
represents one of the crucial differences
between the present work and \cite{JKV}
(see ``Orientation" following Lemma~\ref{Shearer}).

Notice that (by Stirling's Formula)
\beq{mg0r}
\log\Phi_0= \log \frac{n!}{(n/r)! (r!)^{n/r}}
=\frac{n}{r}\log \Cc{n}{r-1}-\gL +O(\log n)
\enq
(recall $\log$ is $\ln$), and that
\begin{equation}\label{gci}
\E \xi_i = \frac{n/r}{\C{n}{r}-i+1} =:\gc_i,
\end{equation}
since in fact
\begin{equation}\label{infact}
\E [\xi_i|A_1\dots A_{i-1}] =\gc_i
\end{equation}
for \emph{any} choice of $A_1\dots A_{i-1}$.
(Strictly speaking \eqref{infact} requires $\Phi(\h_{i-1})\neq 0$, but this will
be true in any case we consider.)
Thus
\begin{equation}\label{Em}
\sum_{i=1}^t\E \xi_i = \sum_{i=1}^t \gc_i=
\frac{n}{r} \log \frac{\C{n}{r}}{\C{n}{r}-t} +o(1),
\end{equation}
provided $\C{n}{r}-t=\omega(n)$.


Let $\mA_t $ be the event
\beq{At}
\mbox{$\left \{\log \Phi_t >\log\Phi_0 - \sum_{i=1}^t\gc_i -o(n)\right\}$.}
\enq

\nin
\emph{Remark.}
We note, perhaps unnecessarily, that \eqref{At} refers to some specific $o(n)$,
so that it makes sense to talk about $\mA_t$ for a particular $n$ (as opposed to a sequence).
Related points will be common below, and, somewhat departing from common
practice, we will elaborate in a couple places where this
seems possibly helpful; see following \eqref{Bi*} for a first instance.

Combining \eqref{mg0r} and the expression for $\sum\gc_i$ in
\eqref{Em} (with $t=T$, in which case the argument of the log is $\C{n}{r}/M$),
we find that $\mA_T$ says
\beq{logPhi}
\log \Phi_T > (n/r)\log (rM/n)-\gL  -o(n),
\enq
which is the same as \eqref{MainIneq};
so Theorem~\ref{ThmX'} is
\beq{Main1}
\pr(\ov{\mA}_T) =  o(1).
\enq
(We will in fact show
$\pr(\cup_{t\leq T}\ov{\mA}_t) =  o(1)$; see \eqref{3terms}.)


For \eqref{Main1}
we use the method of martingales with bounded differences.
Here it is  natural---though we will need a slight variant---to
consider the martingale
\[
\mbox{$\{X_t =\sum_{i=1}^t (\xi_i -\gc_i)\}$}
\]
(it is a martingale by \eqref{infact}), with associated
difference sequence
\[
\{Z_i = \xi_i -\gc_i\}.
\]

In general proving concentration of such $X_t$'s
depends on maintaining some
control over the $|Z_i|$'s,
to which end we keep track of
two sequences of auxiliary
events, $\mB_i$ and $\mR_i$ ($i \le \TTT-1$).
These will be defined in
Section~\ref{BandR}.
Roughly, $\mB_i$ says that no edge of $\h_i$ is in too much more than its
proper share of perfect matchings, while $\mR_i$ consists of
standardish degree restrictions.

For $i\leq T $
it will follow trivially from $\mB_{i-1}$
(see \eqref{proofofxibd}) that
\beq{xibd}
\xi_i =O(\gc_i).
\enq
This is more than enough
for the desired concentration, but can occasionally fail, since $\mB_{i-1}$ may fail.
We accordingly slightly modify the above
$X$'s and $Z$'s, setting
\begin{equation}\label{Zi}
Z_i =\left\{\begin{array}{ll}
\xi_i -\gc_i & \mbox{if $\mB_j$ holds for all $j<i$,}\\
0&\mbox{otherwise}
\end{array}\right.
\end{equation}
(and
$X_t = \sum_{i=1}^tZ_i$).
As shown in Section~\ref{Mart},
a martingale analysis along the lines of
Azuma's Inequality then gives
\begin{equation}\label{conc}
\pr( |X_t| > \gl  ) < n^{-\omega(1)} ~~~\mbox{for $\gl\gg \sqrt{n}$.}
\end{equation}

Notice that if we do have $\mB_i$ for all $i<t< T$
(so $X_t=\sum_{i=1}^t(\xi_i-\gc_i)$)
and
$X_t <\sqrt{n}\log n$
(say; there is plenty of room here),
then we have $\mA_t$;
for
\eqref{xibd} gives
\begin{eqnarray}\label{gc2calc}
\mbox{$\sum_{i=1}^t\xi_i^2$} &=&\mbox{$O(\sum_{i=1}^t \gc_i^2)$} \nonumber\\
&=&
\mbox{$O((n/r)^2\sum\{j^{-2}:j> (n/r)\log n\}) = O(n/\log n);$}
\end{eqnarray}
so (using \eqref{mgt1})
\[
\log \Phi_t
> \log \Phi_0 -\sum_{i=1}^t(\xi_i +\xi_i^2 ) \\
> \log\Phi_0 -\sum_{i=1}^t\gc_i  - O(n/\log n),
\]
where the first inequality uses $\xi_i=o(1)$
(which follows from \eqref{xibd} and \eqref{gci}), and
the $O(n/\log n)$ absorbs the smaller
$\sum_{i=1}^t(\xi_i-\gc_i)=X_t$.

Thus the first failure, if any, of an $\mA_t$ (with $t\leq \TTT$)
must occur either because $X_t$ is too large
or because $\mB_i$ fails for some $i<t$;
formally, we have
\beq{3terms}
\pr(\cup_{t\leq T}\ov{\mA}_t) < \pr(\cup_{t<T}\ov{\mR}_t)  +
\sum_{t< T}\pr(\mA_t\mR_t\ov{\mB}_t)
+
\sum_{t\leq T}\pr((\cap_{i<t}\mB_j)\cap \ov{\mA}_t).
\enq

\nin
Here the last sum is $n^{-\omega(1)}$
by \eqref{conc} and the discussion following it,
and we will show
\beq{Ri}
\pr(\cup_{t< T}\ov{\mR}_t)=o(1)
\enq
and, for $i\leq \TTT$,
\beq{Bi}
\pr(\mA_i\mR_i\ov{\mB}_i) = n^{-\go(1)}.
\enq
So the l.h.s.\ of \eqref{3terms} is $o(1)$, which in particular gives \eqref{Main1}
(and Theorem~\ref{ThmX'}).\qed

\mn
\emph{Remark.}
Thus most of the exceptional probability
comes from the $\mR_i$'s, which include lower bounds on minimum degrees
(see \eqref{Rg2}) whose failure probability is not all that small.
If the process survives the $\mR_i$'s, then the probability that it fails for
some other reason is much smaller.

\mn
\emph{Orientation.}
What this paper is really about---as was \cite{JKV}---is bounding the
increments $\xi_i$; that is, establishing \eqref{xibd}, which, as already
mentioned, follows immediately from $\mB_{i-1}$.
The martingale analysis that handles (\emph{via} \eqref{conc}) the last term in
\eqref{3terms} is then pretty standard, and
the genericity assertions \eqref{Ri} are also fairly routine.

Thus the heart of the matter is \eqref{Bi},
which is proved in Sections~\ref{More}-\ref{PLF1}, with the assistance of the
entropy machinery
developed in Section~\ref{Ent}.  The most important part of this is
Sections~\ref{PLE} and \ref{PLF1},
but the Br\'egman-like bound of Theorem~\ref{TCuckler},
which underlies Section~\ref{PLE} and
seems of independent interest, is also critical:
as mentioned at the beginning of this section, a crucial
difference between the present outline and
the corresponding discussion in \cite{JKV} is the $o(n)$---which in \cite{JKV} was
$O(n)$---in the definition of $\mA_t$, and it is Theorem~\ref{TCuckler} that
provides the opening to exploiting this.

We will try to comment on particular aspects of the argument when we are in
a position to do so more intelligibly.

\mn
\emph{Outline.}
After briefly recalling large deviation basics,
Section \ref{Mart}
records what we need in the way of martingale concentration, in
a form convenient for a second application in Section~\ref{PLF1}, and
gives the calculation for \eqref{conc}.
As mentioned above, Section~\ref{Ent} treats entropy, with main point the
aforementioned Theorem~\ref{TCuckler}.
In Section \ref{BandR} we finally define the
events $\mB_i$ and $\mR_i$ as part of a somewhat more general discussion, give the easy
derivation of \eqref{xibd} from $\mB_{i-1}$, and, in \eqref{Bi*}, slightly reformulate \eqref{Bi}.
The uninteresting proof of \eqref{Ri}
is banished to an appendix that the reader is encouraged to skip.
And, again,
Sections~\ref{More}-\ref{PLF1} prove \eqref{Bi*}, thus establishing \eqref{Bi}
and, according to the above discussion, completing the proof of Theorem~\ref{ThmX'}.

\section{Concentration}\label{Mart}

Before turning to the main business of this section we
review a couple standard
``Chernoff-type"
bounds.
Recall that a r.v.\ $\xi$ is \emph{hypergeometric} if, for some
$s,a$ and $k$,
it is distributed as $|X\cap A|$, where $A$ is a fixed $a$-subset
of the $s$-set $S$ and $X$ is uniform from $\C{S}{k}$.
\begin{thm}
\label{T2.1}
If $\xi $ is binomial or hypergeometric with  $\mathbb{E} \xi  = \mu $, then for $t \geq 0$,
\begin{align}
\Pr(\xi  \geq \mu + t) &\leq
\exp\left[-\mu\varphi(t/\mu)\right] \leq
\exp\left[-t^2/(2(\mu+t/3))\right], \label{eq:ChernoffUpper}\\
\Pr(\xi  \leq \mu - t) &\leq
\exp[-\mu\varphi(-t/\mu)] \leq
\exp[-t^2/(2\mu)],\label{eq:ChernoffLower}
\end{align}
where $\varphi(x) = (1+x)\log(1+x)-x$
for $ x > -1$ and $ \varphi(-1)=1$.
\end{thm}
\nin
(See e.g.\ \cite[Theorems 2.1 and 2.10]{JLR}.)
For larger deviations the following consequence of the finer bound in \eqref{eq:ChernoffUpper}
is helpful.
\begin{thm}
\label{Cher'}
For $\xi $ and $\mu$ as in Theorem~\ref{T2.1} and any $K$,
\begin{eqnarray*}
\Pr(\xi  > K\mu) < \exp[-K\mu \log (K/e)].
\end{eqnarray*}
\end{thm}

\mn

We now turn to martingales and \eqref{conc}.
The argument for the latter
is about the same as that for the corresponding assertion
in \cite{JKV}, but we now present the basic machinery in somewhat greater generality
to support a second application in Section~\ref{PLF1}.
There is nothing much new here, but, lacking a convenient reference, we include
some details.

\begin{lemma}\label{Azish}
If $Z_1\dots Z_t$ is a martingale difference sequence with respect to the
random sequence $Y_1\dots Y_t$
(that is, $Z_i$ is a function of $Y_1\dots Y_i$ and $\E[Z_i|Y_1\dots Y_{i-1}]=0$), then for
$Z=\sum Z_i$ and any $\vt>0$,
\beq{EeeZ}
\mbox{$\E e^{\vt Z} \leq \prod_{i=1}^t \max \E[e^{\vt Z_i}|y_1\dots y_{i-1}]$}
\enq
and, consequently,
for any $\gl>0$,
\beq{Ebound}
\mbox{$\pr(Z>\gl) <e^{-\vt \gl} \prod_{i=1}^t \max \E[e^{\vt Z_i}|y_1\dots y_{i-1}]$}
\enq
(where $y_i$ ranges over possibilities for $Y_i$).
\end{lemma}
\nin

\begin{proof}
As usual, \eqref{Ebound} follows from \eqref{EeeZ},
using
$\pr(Z> \gl) = \pr(e^{\vt Z}>e^{\vt \gl})$ and Markov's Inequality.
For \eqref{EeeZ}, with $B_i$ denoting the $i$th factor on the r.h.s., induction on $t$ gives
\begin{eqnarray*}
\E e^{\vt Z} &=  &
\E\{\E[e^{\vt Z} |Y_1 \dots Y_{t-1}]\}\\
&= &
\E\{e^{\vt (Z_1+\cdots + Z_{t-1})}\E[e^{\vt Z_t}|Y_1 \dots Y_{t-1}]\}\\
&\leq &
B_t\cdot \E[e^{\vt (Z_1+\cdots + Z_{t-1})}]\\
&\leq &
\mbox{$\prod B_i.$}
\end{eqnarray*}
\end{proof}

Both here and in Section~\ref{PLF1}, bounds on the factors in \eqref{EeeZ} are given
by the next observation.
\begin{prop}\label{EeZ}
For a r.v.\ $W \in [0,b]$ with $\E W \leq a$, and $\vt\in [0, (2b)^{-1}]$,
\beq{Ethetas}
\max\{\E e^{\vt (W -\E W )},\E e^{-\vt (W -\E W )}\}
\leq e^{\vt^2 ab}.
\enq
\end{prop}
\begin{proof}
Since the bound is increasing in $a$, it is enough to prove it when $\E W =a$.
Given this and the bounds on $W $, convexity implies that each of
$\E e^{\vt W }$, $\E e^{-\vt W }$ is maximized (for any $\vt$)
when $W $ is $b$ with probability $p:=a/b$ and zero otherwise, in which case we have
\[
\E e^{\vt (W -\E W )} =
e^{-\vt b p}[1- p+p e^{\vt b}], ~~~\E e^{-\vt (W -\E W )} =e^{\vt b p}[1- p+p e^{-\vt b}];
\]
and simple calculations show that
$e^{-xp}[1-p+pe^x]\leq e^{x^2p} $ for $|x|\leq 1/2$ (and any $p$),
implying \eqref{Ethetas}.
\end{proof}

\begin{proof}[Proof of \eqref{conc}]

Let $\vs_i=O(\gc_i)$ be the bound on $\xi_i$ in \eqref{xibd}.
We will apply
Lemma~\ref{Azish} with $Y_i=A_i$ and
$Z_i$ as in \eqref{Zi} (so $Z=X_t$),
using Proposition~\ref{EeZ} with $b=\vs_i$ and $a= \gc_i$
to bound the factors in \eqref{Ebound} (or \eqref{EeeZ}).
(For relevance of the proposition notice that, conditioned on any particular values
of $A_1\dots A_{i-1}$, $Z_i$ is
either identically zero (as happens if $\mB_j$ has failed for some $j<i$) or
$Z_i=\xi_i-\gc_i$, where $\xi_i\in [0,\vs_i]$ has (conditional) expectation
$\gc_i$
(see \eqref{infact}).)
This combination (i.e.\ of Lemma~\ref{Azish} and Proposition~\ref{EeZ})
gives
\[
\mbox{$\pr(X_t > \gl) < \exp[\vt ^2 \sum_{i=1}^t \vs_i \gamma_i -\vt \gl]$}
\]
for any $\gl>0$, provided, say, $\vt\leq 1 $ ($\leq (2\max \vs_i)^{-1}$).
So with
\[
\mbox{$J=\sum_{i=1}^t \vs_i\gamma_i= O(\sum\gc_i^2) = O(n/\log n)$}
\]
(see \eqref{gc2calc}) and $\vt=\min\{1,\gl/(2J)\}$, we have
\[
\Pr(X_t > \gl)<\left\{
\begin{array}{ll}
\exp[-\gl^2/(4J)]&\mbox{if $\gl\leq 2J$,}\\
\exp[-\gl/2]&\mbox{otherwise;}
\end{array}\right.
\]
and for $\gl\gg\sqrt{n}$ (as in \eqref{conc})
each bound is $n^{-\go(1)}$
(in the first case since $J=O(n/\log n)$).

The same argument applies to $\pr(X_t < -\gl)=\pr(-X_t >\gl)$ (though this part of \eqref{conc}
isn't needed for the proof of Theorem~\ref{ThmX'}).
\end{proof}

\section{Entropy}\label{Ent}

Here we develop what we need in the way of entropy.
The main result is Theorem~\ref{TCuckler}, an extension (essentially)
of Theorem~1.2(a) of \cite{Cuckler-K} (itself more or less a generalization of Br\'egman's Theorem
\cite{Bregman})
that is one main point underlying the present improvement of \cite{JKV}.  The discussion also
includes a pair of technical observations,
Lemmas~\ref{TL1} and \ref{TL2}, that support the use of Theorem~\ref{TCuckler} in
Section~\ref{PLC}.

We use $H(X)$ for the \emph{base} $e$
entropy of a discrete r.v.\ $X$;
that is,
\[
H(X) =-\sum_x p(x)\log p(x),
\]
where $p(x) = \pr(X=x)$.
For entropy basics see e.g.\ \cite{Cover-Thomas}.


For a hypergraph $\h$ and $v\in V=V(\h)$ ($=[n]$ as usual), we use
$X(v,\h)$ for the edge containing
$v$ in a uniformly chosen
perfect matching of $\h$, and
$h(v,\h)$ for $H(X(v,\h))$.
(We will not need to worry about $\h$'s without perfect matchings.)

Before turning to our main point we recall one instance of
\emph{Shearer's Lemma} \cite{CFGS}; this played a role in \cite{JKV} corresponding to that of
the present Theorem~\ref{TCuckler}, and we will find some lesser use for it here.
\begin{lemma}\label{Shearer}
For any r-graph $\h,$
\[
\mbox{$\log \Phi(\h)\leq r^{-1}\sum_{v\in V} h(v,\h).$}
\]
\end{lemma}

\nin
\emph{Orientation.}
The main purpose of this section is to recover (essentially) a missing
$-\gL$ ($= -(r-1)n/r$) in the bound of Lemma~\ref{Shearer}.
For example when $r=2$ (so $\gL=n/2$), the lemma bounds
$\log \Phi(G)$
for a $d$-regular, $n $-vertex graph $G$
by $(n /2)\log d$, which an observation of L.\ Lov\'asz and the author
(\cite[Eq.\ (8)]{Cuckler-K} or \cite{Alon-Friedland,Cutler-Radcliffe};
it is just the extension of
Br\'egman to not necessarily bipartite $G$) improves to
$\frac{n }{2d}\log(d!) = (n /2)\log d -\gL +o(n)$.
The missing $\gL$ was irrelevant in \cite{JKV},
since the argument there involved other losses that could not be made smaller
than $O(n)$;
here the present gain will eventually cancel the $-\gL$ in the bound
\eqref{At} of $\mA_t$ (hidden in the $\log \Phi_0$; see \eqref{mg0r}):
see the interplay of
\eqref{logPhi1} and \eqref{logPhi2} in Section~\ref{PLE}.

\mn

In what follows we will treat a p.m.\ $f$ as either a set of edges or
a function from vertices
to $\C{V}{r-1}$; we use $f_v$ for the edge of $f$ containing $v$
(taking the first view) and $f(v)$ for $f_v\sm v$ (taking the second).

For Theorem~\ref{TCuckler}
we consider a random (not necessarily uniform) p.m.\
$\ff$ of
a given $r$-graph $\h$ (with number of vertices divisible by $r$).
We use $v$ for vertices and $Y$ for
$(r-1)$-sets, and always assume $v\not\in Y$.


Set $p_v(Y)=\pr(\ff(v)=Y)$.  For a p.m.\ $f$, let
\[
T(v,f,Y) = \{B\in f: B\neq f_v, B\cap Y\neq \0\}
\]
and $\tau(v,f,Y)= |T(v,f,Y)|$.
Thus $\tau(v,f,Y)\leq r-1$, with equality iff
the vertices of $v\cup Y$ lie in
distinct edges of $f$ (thought of as ``generic" behavior
of $(v,Y)$ w.r.t.\ $f$).  With $f$ running over p.m.s of $\h$, set
\beq{gGvY}
\gG_v(Y) =\{f:\tau(v,f,Y)< r-1\}
\enq
(note this includes $f$'s with $f(v)=Y$)
and $\gc_v(Y)=\pr(\ff\in\gG_v(Y))$.

\begin{thm}\label{TCuckler}
With notation as above,
\begin{eqnarray}\label{fent}
H(\ff) &< &\mbox{$r^{-1} \sum_v H(\ff(v)) -\gL$}
\nonumber\\
&&~~~~~~~~
\mbox{$+ O(\sum_v\sum_Y p_v(Y)\gc_v(Y)^{1/(r-1)})+O(\log n).$}
\end{eqnarray}
\end{thm}
\nin
(Of course when $\ff$ is uniform, $H(\ff(v))$ is another name for $h(v,\h)$.)
Again, the point here is the ``$-\gL$"; the ugly terms following it are errors
we hope to ignore.

\mn
\emph{Proof.}
(The argument here is similar to that for
Theorem~1.2(a) in \cite{Cuckler-K}.)
Note we may assume
$\h=\K$ ($=\C{V}{r}$), since we can regard $\ff$ as a random
matching of $\K$ that doesn't use edges not belonging to $\h$.

We use $f_B$ for the restriction of $f$ (viewed as a function) to $B\sub V$.
For a permutation $\gs$ of $V$---always thought of as an ordering of $V$---and $v\in V$, set
$B(\gs,v)=\{w\in V: \gs(w)< \gs(v)\}$.
Let $\bgs$
be a random (uniform) permutation of $V$ and
$\XX_v = (\bgs,\ff_{B(\bgs,v)})$.
Then (by the ``chain rule" for entropy; see \cite[Theorem 2.2.1]{Cover-Thomas})
\begin{eqnarray}
H(\ff)
&=&\frac{1}{n !} \sum_{\gs}\sum_vH(\ff(v)|\ff_{B(\gs,v)})
\nonumber\\
&=&\sum_v\sum_{\gs}\sum_g \frac{1}{n !}\pr(\ff_{B(\gs,v)}=g)
H(\ff(v)|\gs,g)\nonumber\\
&=& \sum_v H(\ff(v)|\XX_v),\label{HY}
\end{eqnarray}
where $\gs$ ranges over permutations and,
given $\gs$, $g$ ranges over possible values
of $\ff_{B(\gs,v)}$ (and the conditioning on $(\gs,g)$ has
the obvious meaning).

Now let
\[
\bZ_v = \left\{\begin{array}{cl}
\ff_v&
\mbox{if $B(\bgs,v)\cap \ff(v)\neq \0$},\\
(V\sm\{v\})\sm \bigcup \{\ff_w:w\in B(\bgs,v)\}&\mbox{otherwise.}
\end{array}\right.
\]
The condition in the first line just says
$v$ is not the first vertex of $\ff_v$ in $\bgs$,
in which case $\ff_v$ is determined by $\ff_{B(\bgs,v)}$;
these cases will
be basically ignored in what follows.
In the remaining
cases $\bZ_v$ is the set of vertices that can (in principle) belong to $\ff(v)$
once we have specified $\ff(w)$ for $w$ preceding $v$ in $\bgs$.


Since $\bZ_v$ is determined by $\XX_v$, we have
$H(\ff(v)|\XX_v)\leq H(\ff(v)|\bZ_v)$, so, by \eqref{HY},
\beq{HfZ}
H(\ff)\leq \sum_v H(\ff(v)|\bZ_v);
\enq
so we would like to bound $H(\ff(v)|\bZ_v)$.


We now fix $v$ and write $\bZ$ for $\bZ_v$.
We use $Y$ for values of $\ff(v)$
and $Z$ for values of $\bZ$
\emph{not of the form f}, and set
$\pY =p_v(Y)$ and $\gcY=\gc_v(Y)$.
(See the paragraph preceding Theorem~\ref{TCuckler} for the notation.)
We use $\pr(Z)$ for $\pr(\bZ=Z)$,
$\pr(Z|Y)$ for $\pr(\bZ =Z|\ff(v)=Y)$ and so on.
(It may be worth stressing that $\pr$ refers to
$\bgs$ \emph{and} $\ff$, and that these are independent.)

Since $H(\ff(v)|\bZ= f)=0$, we have
\begin{eqnarray}
H(\ff(v)|\bZ)& = &
\sum_Z\pr(Z)\sum_Y\pr(Y|Z)\log \frac{1}{\pr(Y|Z)}\nonumber\\
&=&\sum_Y\sum_Z\pr(Y,Z)\log\frac{\pr(Z)}{\pr(Y,Z)}
\nonumber\\
&=&
\sum_Y\pY \left[\frac{1}{r}
\log \frac{1}{\pY } +\sum_Z \pr(Z|Y)\log\frac{\pr(Z)}{\pr(Z|Y)}
\right]\label{half}\\
&=& r^{-1}H(\ff(v)) +\sum_Y\pY \sum_Z \pr(Z|Y)\log\frac{\pr(Z)}
{\pr(Z|Y)}~.\label{nonneg}
\end{eqnarray}
(For \eqref{half} notice that independence of $\ff$ and $\bgs$ gives
$\sum_Z\pr(Z|Y)=1/r$ for any $Y$ for which $p_Y\neq 0$.)
We would like to show that the second term in \eqref{nonneg}
is less than about $-\gL$.


Fix $Y$ with $p_Y\neq 0$.  Let
$\mS =\{B(\bgs,v)\cap \ff(v)= \0\}$
(that is, $v$ is the first vertex of $f_v$ under $\bgs$)
and for $k\in [n-1]$ set
\[
q_k ~=~\sum\{\pr(Z|Y):Z\supseteq Y, |Z|=k\}
~=~\pr(\mS ,|\bZ|=k|\ff(v)=Y),
\]
\[
r_k ~=~ \sum\{\pr(Z):Z\supseteq Y, |Z|=k\} ~=~  \pr(\mS ,|\bZ|=k,\bZ\supseteq Y).
\]

\mn
(Notice that ``$|\bZ|=k$" and ``$\bZ\supseteq Y$" make sense once we
know $\mS $ holds, and that it is not really necessary to specify
``$Z\supseteq Y$" in the definition of $q_k$.)
Then the inner sum in \eqref{nonneg} is
\beq{rkqk}
\sum_kq_k\sum \left\{\frac{\pr(Z|Y)}{q_k}\log\frac{\pr(Z)}{\pr(Z|Y)}: |Z|=k\right\}
\leq
\sum_kq_k\log\frac{r_k}{q_k}
\enq
(using Jensen's Inequality), so that \eqref{fent} will follow from
\beq{main}
\sum_kq_k\log\frac{r_k}{q_k} ~<~ -(r-1)/r + O(\gcY^{1/(r-1)}+n^{-1}\log n).
\enq

\mn

We next discuss values of the $q_k$'s and $r_k$'s (with justifications to follow).
We have
\beq{qs}
q_k = \left\{\begin{array}{ll} 1/n &\mbox{if $k = r-1,2r-1\dots n-1$,}\\
0&\mbox{otherwise.}
\end{array}
\right.
\enq
For the $r_k$'s we omit precise specification and settle for
upper bounds:  with
\beq{ss}
s_k = \left\{
\begin{array}{ll}
\frac{1}{n }\frac{((k-r+1)/r)_{r-1}}{(n /r-1)_{r-1}}
&\mbox{if $k = r-1,2r-1\dots n -1$,}\\
0&\mbox{otherwise}
\end{array}\right.
\enq
(where $(a)_b = a(a-1)\cdots (a-b+1)$), we have
\beq{rk}
r_k \leq \gcY q_k +(1-\gcY )s_k.
\enq
\begin{proof}[Justification.]

In fact (we assert) \eqref{qs} holds even if we condition on the value of $\ff$;
that is,
\eqref{qs} is still correct if we
replace $q_k$ by
\[
q_k(f)
=\pr(\mS ,|\bZ|=k|\ff=f)
\]
for any $f$ (but we only use this when $f(v)=Y$).
Similarly, we have
\beq{forallf}
\forall f\not\in \gG_v(Y) ~~~s_k(f):=\pr(\mS ,|\bZ|=k,\bZ\supseteq Y|\ff=f) = s_k.
\enq
To see these, first observe that, given $\{\ff=f\}$ and $\mS $, if we order the edges of $f$
according to their first vertices under $\bgs$, then
$|\bZ|-(r-1)$ is $r$ times the number of edges of $f$ that follow
$f_v$, and this number is uniform
from $\{0\dots n/r-1\}$.
(So $q_k(f)$ is as in \eqref{qs}.)

Note further that (again, given $f$)
$Y\sub \bZ$ iff each of the $\tau:=\tau(v,f,Y)$
edges of $T(v,f,Y)$ follows $f_v$.
But once we know $|\bZ|=k$, the set of $f$-edges following $f_v$ is chosen uniformly
from the $((k-r+1)/r)$-subsets of the $(n/r-1)$-set $f\sm \{f_v\}$.
This gives \eqref{forallf} and, more generally (though we won't use it),
\[
s_k(f) =\frac{1}{n }\frac{((k-r+1)/r)_\tau}{(n /r-1)_\tau}
\]
for \emph{any} $f$, $\tau=\tau(v,f,Y)$ and $k$ as in the first line of \eqref{ss}.

Finally, bounding the
second probability in
\[
\mbox{$r_k = \sum_f\pr(f)\pr(\mS ,|\bZ|=k,\bZ\supseteq Y|\ff=f)$}
\]
by $q_k(f)=q_k$ for each $f\in\gG_v(Y)$ gives \eqref{rk}.
\end{proof}


Now returning to \eqref{main} we have, with $\gc=\gcY$ and
$t$ in the sums running from 1 to $n /r$,
\begin{eqnarray}
\sum_kq_k\log\frac{r_k}{q_k} &\leq&
\tfrac{1}{n }\sum \log [\gc + (1-\gc)\tfrac {(t-1)_{r-1}}{(n /r-1)_{r-1}}]
\nonumber\\
&<& \tfrac{1}{n }\sum\log [\gc + (1-\gc)(rt/n )^{r-1}]\nonumber\\
&=&
\tfrac{r-1}{n }\sum \log (rt/n )  +
\tfrac{1}{n }\sum \log (1+\gc\left(\left(\tfrac{n }{rt}\right)^{r-1}  -1)\right)\nonumber\\
&<&
-\tfrac{r-1}{r} + O\left(\tfrac{\log n }{n }\right) +
\tfrac{1}{n }\sum\log \left(1+\gc\left(\tfrac{n }{rt}\right)^{r-1}\right)
\label{penult}
\end{eqnarray}
(using $\sum\log (rt/n ) = \log[(r/n )^{n /r}(n /r)!]$
and Stirling's formula for the last line).

So for \eqref{main} it is enough to bound the sum in \eqref{penult} by
$O(\max\{n\gc^{1/(r-1)}, 1\})$.
For $\gc < (r/n)^{r-1}$
the sum is less than
\[
\sum \log (1+t^{-(r-1)}) < \sum t^{-(r-1)} =O(1)
\]
(recall $r\geq 3$).
For larger $\gc$, we
set $B = \lfloor (n /r)\gc^{1/(r-1)}\rfloor$
(noting that now $\gc  (n/r)^{r-1} < (2B)^{r-1}$) and
bound the sum in \eqref{penult} by
\[   
\sum_{1\leq t\leq B}\log  \left(2\gc\left(\tfrac{n }{rt}\right)^{r-1}\right) +
\sum_{t> B} \gc \left(\tfrac{n }{rt}\right)^{r-1}.
\]    
Here the first sum is less than
\[
B\log (2\gc (n /r)^{r-1}) - (r-1)\int_1^B\log x dx
~~~~~~~~~~~~~~~~~~~~~~~~~~~~~~~~~~~~~~~~~~
\]
\[
~~~~~~~~~
<
B\log (2^rB^{r-1}) - (r-1)[B\log B -B+1]
< B[r\log 2 +r-1] ,
\]
and for the second we have
\begin{eqnarray*}
\gc (n /r)^{r-1} \sum_{t> B} t^{-(r-1) }
&<&
\gc (n /r)^{r-1}\int_{x>B} x^{-(r-1)}dx \\
&<& (2B)^{r-1} \tfrac{1}{r-2}B^{-(r-2)}
\leq 2^{r-1}B.
\end{eqnarray*}
So the sum in \eqref{penult} is $O(B)=O(n\gc^{1/(r-1)})$ as desired.
\qed

\mn

We now turn to the two auxiliary lemmas mentioned at the beginning of this section.
The first of these will help in controlling the error terms in \eqref{fent}
when we come to apply Theorem~\ref{TCuckler}.

\begin{lemma}\label{TL1}
Suppose $p_i,\gc_i \in [0,1]$, $i=1\dots l $, satisfy
\beq{gz1}
l = n \gD ,
\enq
\beq{gz2}
\mbox{$\sum p_i =n ,$}
\enq
\beq{gz3}
\mbox{$\sum p_i\log (1/p_i) > n \log \gD  - O(n )$}
\enq
and
\beq{X}
\mbox{$\sum \gc_i = o(n \gD )$}.
\enq
Then for any nondecreasing $h:[0,1]\ra[0,1]$ with $h(x)\ra 0$ as $x\ra 0$,
\beq{gz5}
\mbox{$\sum p_i h(\gc_i) = o(n ).$}
\enq
\end{lemma}

\begin{proof}
Let $X=\sum \gc_i$ and specify some $\vs$ with $1\gg \vs \gg X/(n \gD ) $.  From \eqref{gz2} we have
\[   
\mbox{$\sum_{\gc_i\leq \vs}p_ih(\gc_i) \leq n h(\vs) =o(n )$},
\]   
so may restrict attention to $i$'s with $\gc_i> \vs$.
Let
$\sum\{p_i:\gc_i>\vs\} =\ga n $.
Then
$  
\mbox{$\sum\{p_ih(\gc_i):\gc_i> \vs\} \leq \ga n $,}
$  
so it will be enough to show
\beq{gasmall}
\ga =o(1).
\enq

Let $T=|\{i:\gc_i>\vs\}|< X/\vs\ll n \gD  $.  We have
\begin{eqnarray}
n \log \gD  -O(n ) &<& \sum p_i\log (1/p_i) \nonumber\\
&<&
(1-\ga )n \log [n \gD /((1-\ga)n )]+\ga n \log [T/(\ga n )]~~~~\label{secondineq}\\
&=& n [\log \gD  + H(\ga) -\ga  \log (n \gD /T)].
\end{eqnarray}
(For \eqref{secondineq} we use the fact that $\sum_{i=1}^l x_i =a$ implies
$\sum x_i\log (1/x_i)\leq a\log (l /a)$.)
But then
$\ga\log(n \gD/T)-H(\ga) =O(1)$
implies \eqref{gasmall}.
(If $\ga\neq o(1)$ then the l.h.s.\ is $\go(\ga)$, implying that $\ga$ \emph{is} $o(1)$.)
\end{proof}

\begin{lemma}\label{TL2}
For  a probability distribution $p=(p_1\dots p_l )$ and
$\mu$ uniform distribution on $[l ]$,
if $H(p)= \log l  -o(1)$, then, with
$\|x\|=\sum |x_i|$,
\beq{TVD}
\|p-\mu\|=o(1);
\enq
equivalently, for some $\vs=o(1)$
and $\B=\{i:p_i\neq (1\pm \vs)/l \}$,

\beq{TVD1}
|\B|=o(l )
\enq
and
\beq{TVD2}
\mbox{$\sum_{i\in \B}p_i = o(1).$}
\enq
\end{lemma}

\begin{proof}
For the equivalence,
note that \eqref{TVD1} and \eqref{TVD2} imply
\[
\mbox{$\|p-\mu\| \leq \sum_{i\in \B}|p_i- 1/l | +\vs
\leq \sum_{i\in \B}(p_i +1/l ) +\vs = o(1),$}
\]
while
\[
\|p-\mu\|\geq \left\{\begin{array}{ll}
\vs |\B|/l ,\\
\sum_{i\in \B}(p_i-1/l )
\end{array}\right.
\]
with (say) $\vs =\|p-\mu\|^{1/2}$ shows that \eqref{TVD} implies \eqref{TVD1} and then \eqref{TVD2}.

\mn

That the hypothesis of the lemma implies \eqref{TVD} is an instance of the
next observation, whose elementary proof we omit.
\begin{prop}\label{calculus}
If $I$ is an interval of $\Re$ and
$f:I\ra\Re$ is twice differentiable with $f(0)=0$ and $f''>0$, then for any
r.v.\ $X$ with $\E X=0$,
\[
\E f(X)=\gO(\min\{\E^2|X|,1\})
\]
(where the implied constant depends on $f$).
\end{prop}

Here, with $\|p-\mu\|=\gd$, we may take $f(x)=(1+x)\log (1+x)$ and
let $X$ be $\ga_i:=l p_i-1$, with $i $ chosen uniformly from $[l]$.
Then, noting that $\E X = \sum p_i-1=0$ and $\E|X|=\gd$, and applying
Proposition~\ref{calculus}, we have
\[
H(p) =\frac{1}{l} \sum(1+\ga_i)\log\frac{l}{1+\ga_i}
=\log l -\E f(X) = \log l -\gO(\gd^2),
\]
which with $H(p) =\log l -o(1)$ implies $\gd=o(1)$.
\end{proof}

\section{Properties $\mA$, $\mR$ and $\mB$}\label{BandR}

Properties here and in later sections are defined for a general $r$-graph $\h$,
and then, for example, the
event $\mA_t$ in Section~\ref{Skeleton} is $\{\mbox{$\bH_t\models \mA$}\}$.
In this section
we use $n$ and $m$ as defaults for the numbers of vertices and edges of $\h$,
so
\beq{nDmr}
nD_\h= mr
\enq
(recall $D$ is average degree).
We will always use
\beq{tm}
t=\Cc{n}{r}-m,
\enq
and
will \emph{tend} to use $A$ for edges and $Z$ or $U$ for general $r$-sets
(members of $\K$).
We assume throughout that we have fixed some positive $\eps$
(it will be essentially the one in Theorem~\ref{ThmX'}), upon which
the implied constants in ``$O(\cdot)$" and ``$\gO(\cdot)$"
depend.

\mn

We say $\h$ has the property $\mA$ (or $\h$ satisfies $\mA$, or $\h\models\mA$) if
\beq{Ag}
\mbox{$\log \Phi(\h) >\log\Phi_0 - \sum_{i=1}^t\gc_i -o(n), $}
\enq
(see \eqref{gci} for $\gc_i$),
and the property $\mR$ if
\beq{Rg1}
\mbox{a.a.\ degrees in $\h$ are asymptotic to $D_\h$,}
\enq
\beq{Rg2}
\gD_\h =O(D_\h), ~~~ \gd_\h =\gO(D_\h),
\enq
and
\beq{Rg3}
\mbox{all codegrees in $\h$ are $o(D_\h)$.}
\enq
As noted above, $\mA_t$ and $\mR_t$
of Section~\ref{Skeleton} are then $\{\bH_t\models \mA\}$
and $\{\bH_t\models \mR\}$.
Note that $\mR$ is ``robust," in that
\beq{robust}
\mbox{if $\h$ satisfies $\mR$ then so does $\h-Z$ for every $Z\in\K$.}
\enq
(We omit the easy justification, just noting that \eqref{Rg2} implies
$D_{\h-Z}\sim D_\h$ and that
each of \eqref{Rg1}-\eqref{Rg3} for $\h-Z$ depends on having \eqref{Rg3} for
$\h$.)

\mn

For $\mB$ a little notation will be helpful.
For  a finite set $S$ and $\ww: S \rightarrow \Re^+$
($:=[0, \infty)$),
set
\[
\overline \ww (S)= |S|^{-1} \sum_{a\in S} \ww(a),
\]
\[
\max \ww (S)= \max_{a \in S}\ww(a),
\]
and
\[
\maxr \ww(S)= \overline \ww (S)^{-1}\max \ww (S).
\]

For $\h \sub \K$ define $\ww_\h:\K\ra \Re^+$ by
\[
\ww_\h (Z) = \Phi(\h -Z),
\]
and say $\h$ has the property $\mB$ if
\[
\maxr \ww_{\h}  (\h ) =  O(1)
\]
(so the number of p.m.s containing any particular $A\in \h$ is not too large
compared to the average).
Then $\mB_t$ in Section~\ref{Skeleton} is $\{\bH_t\models\mB\}$,
and \eqref{Bi} is
\beq{Bi*}
\mbox{for $m>(1+\eps)(n/r)\log n$,
$~~\pr(\bH_{n,m}\models \mA\mR\ov{\mB}) =n^{-\go(1)}.$}
\enq
(More formally:  there is a fixed $C$, depending on the particular $o(\cdot)$'s
and implied constants in $\mA$ and $\mR$, such that
$\pr(\{\bH\models\mA\mR\}\wedge\{\maxr\ww_{\bH}(\bH) > C\})=n^{-\go(1)}$.)

As mentioned at the end of Section~\ref{Skeleton}, \eqref{Bi*}
is shown in Sections~\ref{More}-\ref{PLF1},
and with
\eqref{Ri} (likelihood of the $\mR_i$'s, proved in Section~\ref{Reg})
will complete the proof of Theorem~\ref{ThmX'}.

\mn

We conclude this section with the promised
\beq{proofofxibd}
\mbox{$\mB_{t-1}$ implies
{\rm \eqref{xibd}}.}
\enq
(Recall \eqref{xibd} says $\xi_t=O(\gc_t)$, where $\gc_t$
is now $n/(r(m+1))$; see \eqref{gci} and \eqref{tm}.)
\begin{proof}
Given $\bH_{t-1}=\h$, we have
$\xi_t\leq\max_{A\in \h} \ww_\h(A)/\Phi(\h)$,
while $\gc_t$ is the average of these ratios, since
\[
\mbox{$\sum_{A\in \h}\ww_\h(A)=\Phi(\h)n/r$}
\]
(and $|\h|=m+1$).
This gives \eqref{proofofxibd}.
\end{proof}

\section{More properties}\label{More}

We will get at $\mB$ (and \eqref{Bi*}) \emph{via} several auxiliary properties.
We introduce the first three of these here (there will be a couple more in Section~\ref{PLF1}),
together with
assertions concerning them that together imply \eqref{Bi*}.
Proofs of the assertions are mostly postponed to later sections.


Given $\h$, we now use $D$ for $D_\h$.
The first three auxiliary properties (for $\h$) are:

\begin{itemize}

\item
[$\mC$:] $~~$ if $Z\in \K$ satisfies
\beq{WZ}
\ww_\h(Z) > \Phi(\h)e^{-o(n)},
\enq
then for any $x\in Z$,
\beq{wZxy}
\mbox{$\ww_\h((Z\sm x)\cup y)\more \ww_\h(Z)d(x)/D~$ for a.e.\ $y\in V\sm Z$;}
\enq

\item
[$\mE$:]
$~~~
\ww_\h (A)\sim \Phi(\h)/D$ for a.e.\ $A\in  \h$;

\item
[$\mF$:]
$
~~~\ww_\h(Z)\sim \Phi(\h)/D$ for a.e.\ $Z \in \K$.

\end{itemize}
(More formally, e.g.\ for $\mE$:
there is $\vs=\vs(n)=o(1)$ such that
$|\{A\in \h:  \ww_\h(A)\neq (1\pm \vs)\Phi(\h)/D\}| < \vs|\h|$.)
For perspective on $\mE$ and $\mF$---and for use below---note
that (using \eqref{nDmr})

\beq{wgavg}
(\overline \ww_\h(\h)=) ~ |\h|^{-1}\sum_{A\in \h}\ww_\h(A)
= |\h|^{-1} \Phi(\h)n/r =\Phi(\h)/D.
\enq

\mn

We now use $\bH$ for $\bH_{n,m}$ and
$\h$ for a general $m$-edge $r$-graph on $n$, with $n$ and $m$ as in \eqref{Bi*},
and sometimes write
\[
\mX \Rastar
\mZ
\]
for $\pr(\mX\ov{\mZ})=n^{-\go(1)}$;
so e.g.\ the conclusion of \eqref{Bi*} becomes
\beq{B''}
\{\bH\models \mA\mR\}\Rastar \{\bH\models \mB\}.
\enq


The aforementioned assertions are as follows.

\begin{lemma}\label{LemmaE}
If $\h$ satisfies $\mA\mR$ then it satisfies $\mE$.
\end{lemma}
\begin{lemma}\label{LemmaF}
$\{\bH\models \mA\mR\} \Rastar \{\bH\models \mF\}$.
\end{lemma}
\begin{lemma}\label{LemmaC}
For $x\in Z\in \K$,
\[
\{\bH\models \mR\}\wedge \{\bH-Z\models \mF\}\Rastar \{(\bH,Z,x)\models \eqref{wZxy}\}.
\]
\end{lemma}
\begin{lemma}\label{RCEFB}
If $\h$ satisfies $\mR\mF\mC$ then it satisfies $\mB$.
\end{lemma}

The nonprobabilistic Lemma~\ref{LemmaE}, which is based mainly on the
material of Section~\ref{Ent}, allows us to replace $\mA\mR$ by $\mA\mR\mE$ in
Lemma~\ref{LemmaF}.  That lemma then embodies the idea that $\mE\ov{\mF}$ is unlikely because
the distribution of the $\ww_{\bH}(A)$'s ($A\in \bH$) should reflect
that of the $\ww_{\bH}(Z)$'s ($Z \in \K$).
We regard this natural point, and more particularly \eqref{Qbd} below,
as the heart of our argument; certainly it was the part
whose proof 
took longest to find.

Lemmas~\ref{LemmaE}-\ref{RCEFB}
are stated in the order in which they are used in proving
\eqref{Bi*},
but shown below in ascending order of difficulty and interest.
Thus we prove Lemma~\ref{LemmaC} in Section~\ref{PLC},
Lemma~\ref{LemmaE} in Section~\ref{PLE}, and
Lemma~\ref{LemmaF} in Section~\ref{PLF1}, with
the easy Lemma~\ref{RCEFB} proved here following the derivation of
\eqref{Bi*}.

\begin{proof}[Proof of \eqref{Bi*}]
We first observe that $\{\h\models\mR\}$ implies $\{\h-Z\models\mR\}$ for any $Z$
($\in \K$);
$\{\h\models\mA\}$ implies $ \{\h-Z\models\mA\}$ for any $Z$
as in \eqref{WZ}; and Lemma~\ref{LemmaF} also holds with $\bH-Z$ in place of
$\bH$ (for any $Z$).
The first of these was already noted in \eqref{robust} and the second is trivial,
so we just need the (routine) justification of the third:

With $\bh =|\bH-Z|$ and $m'=(1-\eps/2)m$, Theorem~\ref{T2.1} gives
\beq{bgm'}
\pr(\bh < m') =\exp[-\gO(\eps^2 m)]=n^{-\go(1)}.
\enq
So with $\mX=\mA\mR\ov{\mF}$, we have
\[
\pr(\bH-Z\models \mX) \leq
\pr(\bh< m') +\sum_{h\geq m'}\pr(\bh=h)
\pr(\bH-Z\models \mX|\bh=h),
\]
which is $n^{-\go(1)}$ by \eqref{bgm'} and application of Lemma~\ref{LemmaF}
to the summands.

\mn

We thus have, in addition to Lemma~\ref{LemmaF},
\[
\mbox{for $Z$ as in \eqref{WZ},}~~
\{\bH\models\mA\mR\}\Rastar  \{\bH-Z\models\mF\},
\]
which with Lemma~\ref{LemmaC} gives
\beq{GARGC}
\{\bH\models\mA\mR\}\Rastar \{\bH\models\mC\}
\enq
(since $\{\bH\models\mC\}=\{(\bH,Z,x)\models \eqref{wZxy} ~
\mbox{for all $Z$ as in \eqref{WZ} and $x\in Z$}\}$).
Finally, using Lemma~\ref{RCEFB} with Lemma~\ref{LemmaF} and \eqref{GARGC}
gives \eqref{Bi*} (in the form \eqref{B''}).\end{proof}

\begin{proof}[Proof of Lemma~\ref{RCEFB}.]
We need one more property (again, for a given $\h$):
\begin{itemize}

\item
[$\mD$:]  if $Z_0\in \K$ satisfies \eqref{WZ} then
\beq{Dbd}
\mbox{$\ww_\h(Z)\more \ww_\h(Z_0)D^{-r}\prod_{x\in Z_0}d_\h(x)~$ for a.e.\ $Z\in \K$.}
\enq
\end{itemize}
The next two assertions give Lemma~\ref{RCEFB}.
\beq{DDlemma}
\mbox{If $\h$ satisfies $\mR \mC$ then it satisfies $\mD$.}
\enq
\beq{RDEFB}
\mbox{If $\h$ satisfies $\mR \mD \mF$ then it satisfies $\mB$.}
\enq

For
\eqref{DDlemma} notice that if $\h$ satisfies $\mR \mC$
and $Z_0=\{x_1\dots x_r\}$ satisfies \eqref{WZ},
then induction on $i\in [r]$ shows that
for a.e.\ choice of distinct $y_1\dots y_r\in V\sm Z_0$ we have, with
$Z_i=(Z_{i-1}\sm x_i)\cup y_i$,
\beq{wZi}
\mbox{$\forall i ~~\ww_\h(Z_i)\more \ww_\h(Z_{i-1})d_\h(x_i)/D
\more \ww_\h(Z_0)D^{-i}\prod_{j\leq i}d_\h(x_j)$.}
\enq
(The only thing to observe here is that \eqref{wZi} for $Z_{i-1}$ (with \eqref{WZ} for $Z_0$)
implies \eqref{WZ} for $Z_i$, since $\gd_\h=\gO(D)$ (see \eqref{Rg2})
implies that the r.h.s.\ of
\eqref{wZi} is $\gO(\ww_\h(Z_0))$.)
This gives $\mD$, since it implies that a.e.\ $Z\in \K$ is $Z_r$ for some
$y_1\dots y_r$
supporting \eqref{wZi}.


For \eqref{RDEFB} choose $Z_0\in \K$ with $\ww_\h(Z_0)$ maximum and note $Z_0$
satisfies \eqref{WZ} (since $\ww_\h(Z_0)$ is at least the l.h.s.\ of
\eqref{wgavg}).
Thus
$\mD$ (and $\gd_\h=\gO(D)$) give $\ww_\h(Z)=\gO(\ww_\h(Z_0))$ for a.e.\ $Z\in \K$,
which with $\mF$ implies
$\Phi(\h)/D=\gO(\ww_\h(Z_0))$.  But this gives $\mB$, since
$\overline \ww_\h(\h)=\Phi(\h)/D$
(again see \eqref{wgavg}) and $\max_\h(\h)\leq\ww_\h(Z_0)$.
\end{proof}

\section{Proof of Lemma~\ref{LemmaC}}\label{PLC}

We now use $D$ for $D_{\bH}$ (with $\bH=\bH_{n,m}$) and set
$\bG=\bH-Z$, $Y=Z\sm x$ and $W=V\sm Z$.
Notice to begin that, for any $y\in W$,
\beq{WHYz}
\ww_{\bH} (Y\cup y) = \sum\{\ww_{\bG}(S\cup y):S\in \Cc{W\sm y}{r-1}, ~S\cup x\in \bH\}.
\enq

Let $\bH'=\{A\in \bH:A\cap Z=\{x\}\}$ and $\bH''=\bH\sm \bH'$.
We think of choosing first $\bH''$
(which determines $\bG$) and then $\bH'$.
If $\bH\models \mR$
then (using \eqref{Rg2} for \eqref{DGD} and \eqref{Rg2}-\eqref{Rg3} for \eqref{d'x})
\beq{DGD}
D_{\bG}\sim D,
\enq
\beq{d'x}
d'(x):= |\bH'| ~( =|\{S\in \Cc{W}{r-1}: S\cup x\in \bH\}|)  \sim d_{\bH}(x)=\gO(D),
\enq
and (in view of \eqref{DGD})
$\mF$ for $\bG$ is
\beq{WGZ*}
\mbox{$\ww_{\bG}(U)\sim \Phi' := \Phi(\bG)/D~ ~$
for a.e.\ $~U\in \C{W}{r}$.}
\enq
Thus Lemma~\ref{LemmaC} will follow from
\beq{LCagain}
\pr((\bH,Z,x)\models \eqref{wZxy}|\{\bG\models \eqref{WGZ*}\}\wedge \{d'(x)\models \eqref{d'x}\}) =1-n^{-\go(1)}.
\enq
So we assume $\bH''$ has been chosen so that the conditioning event holds
(note this \emph{is} decided by $\bH''$), and proceed to choosing
$\bH'$.

From \eqref{WGZ*} we have
\beq{WGz}
\mbox{for a.e. $y\in W$, $~\ww_{\bG}(S\cup y)\sim \Phi' ~$ for a.e.\ $~S\in \Cc{W\sm y}{r-1}$;}
\enq
so for \eqref{LCagain} it is enough to show
that if $y$ is as in \eqref{WGz} then the inequality in \eqref{wZxy} holds with probability
$1-n^{-\go(1)}$.  But for such a $y$,
Theorem~\ref{Cher'} (using \eqref{d'x} and $D=\gO(\log n)$)
says that with probability $1-n^{-\go(1)}$,
$\ww_{\bG}(S\cup y)\sim \Phi' $ for all but $o(d'(x))$ of the $S$'s in
\eqref{WHYz}; and whenever this is true we have (as desired)
\[
\ww_{\bH}(Y\cup y) \more \Phi' d'(x) ~~(\sim\ww_{\bH}(Z)d_{\bH}(x)/D).
\]

\section{Proof of Lemma~\ref{LemmaE}}\label{PLE}

Here $\h$ is a general $m$-edge ($n$-vertex) $r$-graph satisfying $\mA\mR$.
We again use $D$ for $D_\h$.

\mn

Setting $p = \left(\C{n}{r}-t\right)/\C{n}{r}$ ($=m/\C{n}{r}$),
and using \eqref{mg0r} and \eqref{Em},
we may rewrite the lower bound in \eqref{Ag} as
\[
\frac{r-1}{r}n \log n  -\gL -\frac{n }{r}\log [(r-1)!]
+ \frac{n }{r} \log p -o(n ),
\]
while
\[
\log D =(r-1)\log n  -\log[(r-1)!] + \log p +o(1).
\]
Thus
$\mA$ for $\h$ says
\beq{logPhi1}
\log \Phi(\h) ~>~ \frac{n }{r} \log D-\gL -o(n )
~>~\frac{1}{r}\sum \log  d(v) -\gL -o(n ),
\enq
the second inequality following from $\sum\log d(v) \leq n \log(\sum d(v)/n )$.
(Here and in the rest of this argument, $v$ runs over vertices and $d(v)$ is $d_\h(v)$.)

On the other hand,
we claim that
\beq{logPhi2}
\log \Phi(\h) <\frac{1}{r} \sum h(v,\h) -\gL +o(n ).
\enq
(Recall from the third paragraph of Section~\ref{Ent}
that $h(v,\h)$ is the entropy of
the edge containing $v$ in a uniform p.m.\ of $\h$.)

\begin{proof}[Proof of \eqref{logPhi2}]
This will follow from Theorem~\ref{TCuckler},
applied with $\ff$ a uniform p.m. of $\h$ (so $H(\ff(v))=h(v,\h)$),
once we show
\beq{on2}
\mbox{$\sum_v\sum_Y p_v(Y)\gc_v(Y)^{1/(r-1)}=o(n ).$}
\enq
(Recall from the passage preceding Theorem~\ref{TCuckler} that $p_v(Y)=\pr(\ff(v)=Y)$
and $\gc_v(Y)$ is the probability that fewer than $r$ edges of $\ff$ meet $Y\cup v$.
Of course here, for a given $v$, the only relevant $Y$'s are those with $Y\cup v\in \h$,
and we restrict to these in the following discussion.)

For \eqref{on2} we apply Lemma~\ref{TL1} with $h(x)=x^{1/(r-1)}$,
$i$ running over pairs $(v,Y)$, and, for $i=(v,Y)$, $p_i=p_v(Y)$ and
$\gc_i=\gc_v(Y)$; thus $l=\sum d(v)$ ($=nD$)
and $\gD=D$.
Then \eqref{on2} becomes \eqref{gz5}, so we need
\eqref{gz1}-\eqref{X}.
The first two of these are immediate and the third follows
from \eqref{logPhi1} \emph{via} Lemma~\ref{Shearer}:
\begin{eqnarray*}
\sum p_i\log(1/p_i)&=& \sum \sum p_v(Y)\log (1/p_v(Y)) \\
&=&\sum h(v,\h) ~\geq ~r\log\Phi(\h)~> ~n \log \gD-O(n )
\end{eqnarray*}
(with the first inequality given by Lemma~\ref{Shearer} and the second by the first
part of \eqref{logPhi1}).

For \eqref{X} we use \eqref{Rg3}: with $\kappa$ ($=o(D)$)
the largest codegree in $\h$
(and $\gG_v(Y)$ as in \eqref{gGvY}), we have
\[
\mbox{$\sum\sum \gc_v(Y) = \sum_f\pr(\ff=f) |\{(v,Y):f\in \gG_v(Y)\}|
\leq \frac{n }{r}\C{r}{2}\kappa r =o(nD )$}
\]
(where the third expression bounds each of the cardinalities in the preceding sum,
since $f\in \gG_v(Y)$ iff the edge $Y\cup v$ meets some member
of $f$ more than once).

(For our random $\h$---as opposed to one just assumed to satisfy $\mA$ and $\mR$---this last
bit is particularly crude since most codegrees will be \emph{much} smaller than $\kappa$.)

\end{proof}

Now combining \eqref{logPhi1} and \eqref{logPhi2} we have
\[
\sum h(v,\h) > \sum \log d(v) -o(n ),
\]
implying (note $h(v,\h)\leq \log d(v)$ is trivial)
\beq{aay}
\mbox{$h(v,\h) > \log d(v)-o(1)~$ for a.e.\ $v$.}
\enq
But
Lemma~\ref{TL2} says that for any $v$ as in \eqref{aay}
there is a set of $(1-o(1))d(v)$ edges $A$ at $v$
with $\ww_\h(A) = (1\pm o(1))\Phi(\h)/d(v)$
(note $p_v(Y)=\ww_\h(Y\cup v)/\Phi(\h)$),
and combining this with \eqref{Rg1} gives $\mE$.

\section{Proof of Lemma~\ref{LemmaF}}\label{PLF1}

We again use $\h$ for a general $m$-edge $r$-graph,
$\bH=\bH_{n,m}$ and $D = mr/n$ ($=D_\h=D_{\bH}$).

\mn

By Lemma~\ref{LemmaE}, Lemma~\ref{LemmaF} is the same as
\beq{L7.2''}
\{\bH\models \mA\mR\mE\} \Rastar \{\bH\models \mF\}.
\enq
Note that, in view of this, we may assume
\beq{sillyp}
m = |\K|-\gO(|\K|),
\enq
since otherwise
$\mE$ and $\mF$ are equivalent and \eqref{L7.2''} is vacuous.
(This rather silly point will be needed for \eqref{Ubd}.)

It will be convenient to further reformulate as follows.
For any $\h$ set
\[
\ga(\h) =\inf\{\ga:|\{U\in \K: \ww_\h(U)\neq (1\pm \ga)\Phi(\h)/D_\h\}|< \ga |\K|\}.
\]
Then $\{\h\models \mF\} = \{\ga(\h)=o(1)\}$ and \eqref{L7.2''} is equivalent
to\footnote{With $\mG=\{\bH\models \mA\mR\mE\} $ and $\mH(\nu) =\{\ga(\bH)> \nu\}$,
\eqref{L7.2''} says
\[
\mbox{there is $\vs=o(1)$ such that $ \pr(\mG\wedge\mH(\vs))=n^{-\go(1)}$,}
\]
while \eqref{L7.2'}  implies
\[
\forall k,  ~ \pr(\mG\wedge\mH(1/k))<n^{-k}~~\mbox{for $n\geq n_k$};
\]
and we get the former from the latter by taking $\vs(n)  =(\max\{k:n_k\leq n\})^{-1}$.}
\beq{L7.2'}
\mbox{for any fixed $\theta>0$, $~\pr(\{\bH\models \mA\mR\mE\}\wedge \{\ga(\bH)>2\theta\}) =n^{-\go(1)}$.}
\enq
(The $2\theta$ will be convenient below.)
So for the rest of this section we fix $\theta>0$ and aim for \eqref{L7.2'}.


Set
\[
\Phi'=\Phi(\bH)/D.
\]
Notice that $\{\bH\models \mE\}\wedge \{\ga(\bH)>2\theta\}$
implies
\begin{itemize}
\item[$\mQ$:]
\emph{$\ww_{\bH}(A)\sim \Phi'$ for a.e.\ $A\in  \bH$, but
$\ww_{\bH}(\uu )\neq (1\pm 2\theta) \Phi'$ for
at least a $(2\theta)$-fraction of the $\uu $'s in $ \K\sm\bH$.}

\end{itemize}
So for \eqref{L7.2'} it is enough to show
\beq{Qbd}
\pr(\bH\models\mA\mR\mQ) <n^{-\go(1)}.
\enq

For the proof of this we work with
an auxiliary random set $\bT$
chosen uniformly from $\C{\bH}{\tau}$, where $\tau$, which will be specified later
(see the paragraph containing \eqref{gznu}-\eqref{param4}), will at least satisfy
\beq{tau}
\log n \ll \tau \ll \log^2n.
\enq

\nin
We take $\bF=\bH\sm \bT$ and
\[
\gz= e^{-\tau/D},
\]
and will be interested in a property of the pair $(\bH,\bT)$ (or $(\bF,\bT)$),
\emph{viz.}
\begin{itemize}
\item[$\mV$:]
\emph{$\ww_{\bF}(A)\sim \gz \Phi'$ for a.e.\ $A\in   \bT$, but
$\ww_{\bF}(\uu )\neq (1\pm \theta) \gz \Phi'$ for
at least a $\theta$-fraction of the $\uu $'s in $ \K\sm\bH$.}
\end{itemize}
(Note
$\gz \ww_{\bH}(\uu )$ is a natural
asymptotic value for $\ww_{\bF}(\uu )$ since
each p.m.\ of $\bH-U$ survives in $\bF$ with probability
about $(1-\tau/m)^{n/r-1}\sim \gz$; \emph{cf.} \eqref{1-vs}.)

\mn

Here we exploit the familiar leverage derived from the interplay
of two natural ways of generating the pair $(\bH,\bT)$:

\begin{itemize}
\item[(A)]
choose $\bH$ and then $\bT$
(as above);

\item[(B)]
choose $\bF $ and then $\bT$ (determining $\bH=\bF\cup\bT$).

\end{itemize}

Now writing simply $\mA$ for $\{\bH\models \mA\}$ and similarly for $\mR$ and $\mQ$,
and $\mV$ for $\{(\bH,\bT)\models \mV\}$, we will show
\beq{U|Q}
\pr(\mV|\mA\mR\mQ) > 1-o(1)
\enq
and
\beq{Ubd}
\pr(\mV) = n ^{-\go(1)}.
\enq
These give \eqref{Qbd}, since
\[
\pr(\mA\mR\mQ) =\pr(\mA\mR\mQ\mV)/\pr(\mV|\mA\mR\mQ)\leq
\pr(\mV)/\pr(\mV|\mA\mR\mQ).
\]
(So \eqref{U|Q} is more than is needed here.)
We first dispose of the easier \eqref{Ubd}.

\begin{proof}[Proof of \eqref{Ubd}]
Here we use viewpoint (B).  The (natural) idea is:  $\bF$ determines the
weights $\ww_{\bF}(\uu )$ (for all $U\in \K$, though here we are only interested in $U\in \K\sm \bF$),
and $\mV$ then requires that $\bT$ be (pathologically) drawn almost entirely from $U$'s with weights
close to $\gz \Phi'$, though this group excludes an $\gO(1)$-fraction of $\K\sm \bF$.

A small complication is that $\bF$ doesn't determine $\Phi'$.
Among several ways of dealing with this, the following seems nicest.
Given $\bF$,
let $\uu _1,\ldots$ be an ordering of $\K\sm \bF$ with
$\ww_{\bF}(\uu _1)\leq \ww_{\bF}(\uu _2)\leq \cdots$,
and let $\Y$ and $\ZZZ$ be (resp.) the first and last $\theta|\K\sm \bF|/3$ of the $\uu _i$'s.
Then, \emph{whatever} $\Phi'$ turns out to be, the second part of $\mV$ requires that
at least one of $\Y$, $\ZZZ$ be contained in
\[
\W:=\{\uu : \ww_{\bF}(\uu ) \neq (1\pm \theta)\gz \Phi'\}
\]
(or \eqref{sillyp}, with $\tau\ll m$, implies
$|\W\sm\bH|< |\Y|+|\Z| < 2\theta(|\K\sm \bH|+\tau)/3<
\theta |\K\sm \bH|$).
But if this is true then the first part of $\mV$ requires that (say)
\beq{minYZ}
\min\{|\bT\cap \Y|,|\bT\cap \ZZZ|\}< \theta\tau/4;
\enq
and, since
\[
\E |\bT\cap \Y|=\E|\bT\cap \ZZZ| =\theta\tau/3
\]
(and $\theta$ is fixed), Theorem~\ref{T2.1} bounds the probability of \eqref{minYZ}
by
$e^{-\gO(\tau)}$, which is $n^{-\go(1)} $ by \eqref{tau}.\end{proof}

\begin{proof}[Proof of \eqref{U|Q}]

We now need to pay some attention to parameters.
We first observe that if $\h\models\mA\mR$, then there is
$\gc=o(1)$ (depending on the $o(n)$ in $\mA$ and, in $\mR$,
the (explicit or implicit)
$o(\cdot)$'s in \eqref{Rg1} and \eqref{Rg3},
and the implied constants in \eqref{Rg2}),
such that for each $U\in \K$ with (say)
\beq{Phi*bd}
(\ww_\h(U)=) ~~\Phi(\h-U)> \Phi(\h)n^{-r},
\enq
$\h^*:=\h-U$ and $\Phi^*:=\Phi(\h^*)$
satisfy
\beq{heavysum}
\sum\{\ww_{\h^*}(A):A\in \h^*, \ww_{\h^*}(A)\neq (1\pm \gc)\Phi^*/D\}< \gc n\Phi^*.
\enq
To see this, notice that each relevant $\h^*$ satisfies $\mA\mR$
(see \eqref{robust} for $\mR$),
so also $\mE$ by Lemma~\ref{LemmaE}.
But then $\h^*$ contains $(1-o(1))|\h^*|\sim nD/r$ edges of $\ww_{\h^*}$-weight
$(1\pm o(1))\Phi^*/D_{\h^*}\sim \Phi^*/D$,
with (both) asymptotics following easily from $\h\models\mR$
(see \eqref{Rg2});
so such edges account for all but a $o(1)$-fraction of the total weight $\Phi^*(n-r)/r\sim \Phi^*n/r$.
This gives \eqref{heavysum} for a suitable $\gc=o(1)$.

\mn

We now choose $\tau =\nu \log n$---noting that then
\beq{gznu}
\gz ~~(=e^{-\tau/D}) ~ > e^{-\nu}
\enq
(since $D> \log n$)---together with $M$ and $\eta$,
satisfying
\beq{param1}
\log n\gg\nu\gg 1
\enq
(which is \eqref{tau});
\beq{param2}
e^{-\nu}\gg \gc;
\enq
\beq{param3}
\tau\gg M\left\{\begin{array}{ll}
\gg \gc \tau,\\
> 1+\gc;
\end{array}\right.
\enq
and
\beq{param4}
e^{-\nu}\gg \eta\gg \sqrt{\tau M}/\log n.
\enq
Note this is possible:
we may choose $\nu\ra\infty$ as slowly as we like (which in particular gives
\eqref{param1} and \eqref{param2}); we then want to choose $M$ as in \eqref{param3} satisfying (to
leave room for $\eta$)
$
e^{-\nu}\gg \sqrt{\tau M}/\log n;
$
and this is possible if $e^{-\nu}\gg \max\{\nu\sqrt{\gc},\sqrt{\nu/\log n}\}$,
which is true for a slow enough $\nu$.

\mn

For the proof of \eqref{U|Q} we use viewpoint (A) (choose $\bH$, then $\bT$).
We assume we have chosen $\bH=\h$ satisfying $\mA\mR\mQ$; so $\pr$ now
refers just to the choice of $\bT$, and \eqref{U|Q} will follow from
\beq{G,T}
\pr((\h,\bT)\models \mV) = 1-o(1).
\enq
It will be enough to show that for $\uu \in \K$ as in \eqref{Phi*bd}
(i.e.\
$\ww_\h(\uu )> \Phi(\h)n^{-r}$),
\begin{eqnarray}
\pr(\ww_{\bF}(\uu ) \sim \gz\ww_\h(\uu )) = 1-o(1)&\mbox{if $U\in \K\sm \h$,}
\label{WGTZ1}\\
\pr(\ww_{\bF}(U ) \sim \gz\ww_\h(U )|U\in \bT) = 1-o(1)&\mbox{if $U\in \h$.}
\label{WGTZ2}
\end{eqnarray}

\mn

Before proving this we show that it does give \eqref{G,T}.
If $\h$ satisfies $\mQ$ then for a suitable $\vs=o(1)$,
\beq{GO}
|\{A\in \h:\ww_\h(A)\neq (1\pm \vs)\Phi'\}| \ll |\h|.
\enq
Thus, with $\h^0$ the set in \eqref{GO}, we have
$
\E |\bT\cap\h^0|=\tau |\h^0|/|\h|\ll \tau,
$
so
\[
\mbox{$|\bT\cap \h^0|\ll\tau~$  w.h.p.}
\]
(by Theorem~\ref{T2.1} or just Markov's Inequality).
But for the first part of $\mV$ to fail we must have either
$|\bT\cap\h^0|=\gO(\tau)$,
which we have just said occurs with probability $o(1)$, or
\[   
|\{A\in \bT\sm\h^0: \ww_{\bF}(A)\not\sim\gz\ww_\h(A)\}|=\gO(\tau),
\]   
which has probability $o(1)$ by \eqref{WGTZ2} (and Markov).

Similarly, failure of the second part of $\mV$ implies
\beq{WZLarge}
\ww_{\bF}(U) = (1\pm \theta)\gz\Phi'
  \not\sim \gz\ww_\h(U)
\enq
for at least $\theta|\K\sm \h|$ of those
$\uu $'s in
the second part of $\mQ$ that satisfy
\beq{WZlarge}
\ww_\h(\uu )> (1-\theta)\gz \Phi' > n^{-o(1)}\Phi(\h)/D
\enq
(since those failing \eqref{WZlarge} \emph{cannot} satisfy \eqref{WZLarge};
for the second bound in \eqref{WZlarge} see \eqref{gznu} and \eqref{param1}).
But since the bound in \eqref{WZlarge} is larger than the one in \eqref{Phi*bd},
\eqref{WGTZ1} implies that the probability that \eqref{WZLarge}
holds for such a set of $U$'s is $o(1)$.

\mn

Finally, we prove \eqref{WGTZ1}; the proof of \eqref{WGTZ2} is almost
literally the same and is omitted.
(Note the probability in \eqref{WGTZ2} is just
$\pr(\ww_{\h\sm \bT_0}(U ) \sim \gz\ww_\h(U ))$, with $\bT_0$ uniform from
$\C{\h\sm\{U\}}{\tau-1}$.)

\begin{proof}[Proof of \eqref{WGTZ1}.]

We now fix $U $ as in \eqref{Phi*bd} (and recall $\h^*=\h-\uu $ and
$\Phi^*=\Phi(\h^*)$).

\mn

Say $A\in\h$ is \emph{heavy} if $A\in\h^*$ and $\ww_{\h^*}(A)> M\Phi^*/D$, and note that by
\eqref{heavysum} (and $M> 1+\gc$ from \eqref{param3}),
\beq{fewheavies}
\mbox{the number of heavy edges in $\h$ is less than $\gc nD/M=\gc mr/M$,}
\enq
implying
\beq{Theavy}
\pr(\mbox{$\bT$ contains a heavy edge}) < \gc \tau r/M =o(1)
\enq
(see \eqref{param3}).  So it is enough to show \eqref{WGTZ1} conditioned on
\beq{noheavy}
\{\mbox{$\bT$ contains no heavy edges}\}.
\enq
We will instead show a slight variant, replacing
$\bT$ by $\bT'=\{A_1\dots A_\tau\}$, with the $A_i$'s chosen uniformly and \emph{independently}
from the non-heavy edges of $\h$; thus:
\beq{WGTZ'}
\pr(\ww_{\h\sm\bT'}(\uu ) \sim \gz \ww_\h(\uu )) =1-o(1).
\enq
Of course this suffices:  we may couple $\bT$ (conditioned on \eqref{noheavy})
and $\bT'$ so they agree whenever the
edges of $\bT'$ are distinct, which occurs w.h.p.\ (more precisely, with probability at least
$1-\tau^2/m$), and the probability in \eqref{WGTZ1} is then at least
the probability in \eqref{WGTZ'} minus $\pr(\bT'\neq\bT)$.

\mn

For the proof of \eqref{WGTZ'}, let
\[
X= X(A_1\dots A_\tau)  =\Phi(\h^*\sm\{A_1\dots A_\tau\})= \ww_{\h\sm \bT'}(U).
\]
Since $\eta\ll \gz $ (see \eqref{gznu} and
\eqref{param4}), \eqref{WGTZ'} will follow from
(recall $\ww_\h(U)=\Phi^*$)
\beq{EPhiX}
\E X\sim \gz \Phi^*
\enq
and
\beq{prXE}
\pr(|X-\E X|>\eta \Phi^*)=o(1).
\enq
\begin{proof}[Proof of \eqref{EPhiX}]
Let $M_i$ run through the p.m.s of $\h^*$ and let $x_i$ be the number of
heavy edges in $M_i$.
Then with $m'$ the number of non-heavy edges in $\h$, we have
\[
\mbox{$\E X = \sum_i (1-(n/r-1-x_i)/m')^\tau$}
\]
and, by \eqref{heavysum},
\beq{sumxi'}
\sum x_i =\sum\{\ww_{\h^*}(A):\mbox{$A\in \h^*$, $A$ heavy}\}<\gc n\Phi^*.
\enq
These imply, with $\varrho =(n/r-1)/m'$,
\begin{eqnarray}
(1-\varrho)^\tau \Phi^*&\leq & \mbox{$\E X
~<~ \sum_i e^{-(\varrho-x_i/m') \tau}$}\nonumber\\
&< &\left[e^{-\varrho\tau}+ \gc rn/(n-r)\right]\Phi^*.\label{EXsum}
\end{eqnarray}
Here the last inequality follows from \eqref{sumxi'} and convexity of the
exponential function, which imply that the sum in \eqref{EXsum} is at most what it would be
with $\frac{\gc n\Phi^*}{n/r-1} =\frac{\gc rn\Phi^* }{n-r}$ of the $x_i$'s equal to $n/r-1$
and the rest (the number of which we just bound by $\Phi^*$) equal to zero.

In view of \eqref{EXsum}, \eqref{EPhiX} will follow from
\beq{1-vs}
(1-\varrho)^\tau\sim e^{-\varrho\tau}\sim e^{-\tau/D} ~~(=\gz)
\enq
(and $\gc \ll \gz$, which is given by \eqref{gznu} and \eqref{param2}).
For the two parts of \eqref{1-vs} we need (resp.) $\varrho^2\ll 1/\tau$ and
$|\varrho-1/D|\ll 1/\tau$.  The first of these follows from \eqref{fewheavies}
(which gives $m'\sim m$, though here $m'=\gO(m)$ would suffice)
and \eqref{tau}.
For the second, now using \eqref{fewheavies} more precisely
 (and recalling $D=mr/n$), we have
\[
\left| \frac{n/r-1}{m'}-\frac{n/r}{m}\right|
\leq \frac{1}{m'}+\frac{n}{r}~\frac{m-m'}{mm'} < \frac{1}{m'}+\frac{n}{r}\frac{\gc r}{Mm'}\ll \frac{1}{\tau},
\]
with the last inequality a (weak) consequence of \eqref{param3}.
\end{proof}

\begin{proof}[Proof of \eqref{prXE}]

We
consider the (Doob) martingale
\beq{Doob}
X_i=X_i(A_1\dots A_i)=\E[X|A_1\dots A_i]
~~~(i=0\dots \tau),
\enq
with difference sequence $Z_i=X_i-X_{i-1}$
($i\in [\tau]$)
and $Z =\sum Z_i$ ($=X-\E X$).
For the next little bit we use $\E_S$ for expectation
with respect to
$(A_i:i\in S)$.

Given $A_1\dots A_{i-1}$ we may express
\beq{ZW}
Z_i=\E W-W,
\enq
where $\E$ refers to $A$ chosen uniformly from the non-heavy edges of $\h$ and
\[
W(A) =\E_{[i+1,\tau]}\Phi(\h^*\sm\{A_1\dots A_{i-1},A_{i+1}\dots A_\tau\})~~~~~~~~~~~~~~~~~~~
\]
\beq{W(A)}
~~~~~~~~~~~~~~~~~~~~~~ -\E_{[i+1,\tau]}\Phi(\h^*\sm\{A_1\dots A_{i-1},A,A_{i+1}\dots A_\tau\}).
\enq

\mn
For \eqref{ZW} just notice that
\[
X_i(A_1\dots A_{i-1},A) =\E_{[i+1,\tau]}\Phi(\h^*\sm\{A_1\dots A_{i-1},A,A_{i+1}\dots A_\tau\}),
\]
while
\[
X_{i-1}(A_1\dots A_{i-1}) =\E_{[i,\tau]}\Phi(\h^*\sm\{A_1\dots A_{i-1},A_i,A_{i+1}\dots A_\tau\}).
\]
(The first term on the r.h.s.\ of \eqref{W(A)}, which is chosen to give \eqref{WAWA},
doesn't depend on $A$ so doesn't affect \eqref{ZW}.)

We also have
\beq{WAWA}
0\leq W(A) \leq ~\ww_{\h^*}(A),
\enq
since these bounds hold even if we remove the $\E$'s in \eqref{W(A)}.
Thus $W$ satisfies the conditions in Proposition~\ref{EeZ} with $b=M\Phi^*/D$ and $a=\Phi^*/D$
(the latter since
$|\h|^{-1}\sum_{A\in \h}\ww_{\h^*}(A) =|\h|^{-1}\Phi^*(n/r-1) <\Phi^*/D$---note $\ww_{\h^*}(A):=0$
if $A\not\in \h^*$---and averaging instead only over non-heavy edges can only decrease this).
So for any
\beq{gzbds}
\vt \in [0,(2b)^{-1}],
\enq
we may apply Lemma~\ref{Azish} to each of $Z$, $-Z$, using Proposition~\ref{EeZ}
(with \eqref{ZW}) to bound the factors in
\eqref{EeeZ}, yielding
\[
\max\{\E e^{\vt Z},\E e^{-\vt Z}\}\leq e^{\tau \vt^2 ab} = \exp[\tau\vt^2 M(\Phi^*/D)^2]
\]
and, for any $\gl>0$,
\beq{maxprZ}
\max\{\pr(Z>\gl),\pr(Z<-\gl)\} < \exp[\tau\vt^2 M(\Phi^*/D)^2-\vt \gl].
\enq

For \eqref{prXE} we use \eqref{maxprZ} with $\gl =\eta \Phi^*$ and
\beq{gzvals}
\vt =\min\left\{\tfrac{\eta\Phi^*}{2\tau M(\Phi^*/D)^2}, \tfrac{D}{2M\Phi^*}\right\}
=\tfrac{D}{2M\Phi^*}\min \left\{ \tfrac{\eta D}{\tau},1\right\}
\enq
(the first value in ``min"
minimizes the r.h.s.\ of \eqref{maxprZ}
and the second enforces \eqref{gzbds}), and should show that the exponent in \eqref{maxprZ}
is then $-\go(1)$.

Suppose first that $\eta D\leq \tau$, so $\vt$ takes the first value(s) in \eqref{gzvals}.
Then the negative of the exponent in \eqref{maxprZ} is (using \eqref{param4} and $D\geq \log n$)
\[
\frac{(\eta\Phi^*)^2}{4\tau M(\Phi^*/D)^2} = \frac{\eta^2D^2}{4\tau M}
=\go( 1).
\]
If instead $\eta D>\tau$,
then $\vt = D/(2M\Phi^*)$ and the exponent in \eqref{maxprZ} is
\[
\frac{D^2}{(2M\Phi^*)^2}\tau M\left(\frac{\Phi^*}{D}\right)^2 -\frac{D}{2M\Phi^*}\eta \Phi^*
=
\frac{\tau}{4M}-\frac{\eta D}{2M}
< -\frac{\tau}{4M} = -\go(1),
\]
where we used the assumed $\eta D>\tau$ and, from \eqref{param3}, $\tau\gg M$.
\end{proof}\end{proof}\end{proof}

\section{Reduction}\label{Reduction}

In this section we derive Theorem~\ref{ThmY'} from the following statement, which
will be proved in \cite{HT}.
\begin{thm}\label{ThmZ}
Fix a small positive $\eps$ and suppose $\gd_x \sim \eps\log n$
for each $x\in W :=[n]$.
Let $M\sim (n/r)\log n$ and let $\bH^*$ be distributed as
$\bH_{n,M}$ conditioned on
\beq{firstL}
\{d_{\bH}(x)\geq \gd_x ~\forall x\in W\}.
\enq
Then w.h.p.
\[
\Phi(\bH^*)> \left[e^{-(r-1)}\log n\right]^{n/r}e^{-o(n)}.
\]
\end{thm}
\nin
In other words: for $\vs\ll 1$ there is $\varrho\ll 1$ such that if
$M=(1\pm \vs)(n/r)\log n$ and $\gd_x=(1\pm \vs)\eps \log n$ for each $x$,
then
\[
\Pr\left(\Phi(\bH^*)\leq\left[e^{-(r-1)}\log n\right]^{n/r}e^{-\varrho n}\right)
< \varrho.
\]
(The $n$ here will not be exactly the one in Theorem~\ref{ThmY'}, and will
be renamed $n'$ when we come to use it.)

\mn

For the rest of this section $\bH_T$ is as in Theorem~\ref{ThmY'}.
Our discussion through Lemma~\ref{stoppingLemma} is
adapted from \cite{DK}.

We employ the following standard device for handling the process $\{\bH_t\}$
of Theorems~\ref{ThmY} and \ref{ThmY'}.
Let $\xi_A$, $A\in \K$, be independent random variables, each uniform from $[0,1]$,
and 
set $\bG_\gl=\{ A \in \K\ : \ \xi_{A} \leq \gl \}$.
Members of $\bG_\gl$ are $\gl$-\emph{edges} and we use $d_\gl$ for degree in $\bG_\gl$.
Of course with probability one the $\xi_A$'s are
distinct.
If they \emph{are} distinct---which we assume henceforth---they define the discrete process
$\{\bH_t\}$ in the natural way (add edges $A$ in the order in which the
$\xi_A$'s appear in $[0,1]$).

Fix a small positive $\eps$.
Let $\gdz=\lfloor\eps \log n\rfloor$,
let $g$ be a suitably slow $\go(1)$, and set:
\[\gL =\min\{\gl: \mbox{$\bG_\gl$ has no isolated vertices}\}\]
(so $\bH_T=\bG_\gL$);
\[
\gs = \frac{\log n - g(n)}{\Cc{n-1}{r-1}} ~~\text{and}~~\beta = \frac{\log n + g(n)}{\Cc{n-1}{r-1}};\]
\[
W_{\gs} = \{v \in [n] \ : \ d_\gs (v) < \gd_0\};
\]
and
\[
\mbox{$Y=W_\gs \cup\bigcup\{A:A\in \bG_\gb, A\cap W_\gs\neq \0\}$}.
\]
Parts (b) and (c) of the next lemma are
(a) and (c) of Lemma 5.1 in  \cite{DK}.
\begin{lemma}\label{stoppingLemma}
With the above setup, w.h.p.

\mn
{\rm (a)} $|W_\gs| < n^{2\ga}$, with $\ga\sim \eps\log (e/\eps)$;

\mn
{\rm (b)} $\gL\in (\gs, \beta)$;

\mn
{\rm (c)} in $\bG_\gb$, no edge meets $W_{\gs}$ more than once
and no $u\not\in W_{\gs}$ lies in more than one edge meeting
$Y \setminus \{u\}$.

\end{lemma}

\nin
\emph{Remarks.}
Once $\gL<\gb$ as in (b), the initial $W_\gs$ in the definition of $Y$
is superfluous.
For Theorem~\ref{ThmY}, $|W_\gs|=o(n)$ in (a) would suffice,
but for Theorem~\ref{ThmY'} we need a little more
(precisely, $|W_\gs|=o(n/\log\log n)$), to make
the bound on $\Phi$ in Theorem~\ref{ThmZ} (in which, again, $|W|$ will
not be exactly the present $n$) an
instance of \eqref{TY'bd};
see \eqref{nvsn'}.

\begin{proof}
[Proof of {\rm (a)}.]
Since (for any $v$) $d_\gs(v)$ is binomial with mean $\mu:=\C{n-1}{r-1}\gs\sim \log n$, Theorem~\ref{T2.1} gives
\beq{prvWgs}
\pr(v\in W_\gs) < \exp[-\mu\vp (-(1-\eps-o(1)))]  < n^{-1+\ga},
\enq
with
$\ga$ as in (a); and (a) then follows \emph{via} Markov's Inequality.
\end{proof}

\mn

It will now be convenient to fix some linear ordering ``$\prec$" of $\K$.
Choose (and condition on)

\mn
\[
W_\gs,
\]
\beq{xiAs'}
\{A\in \bG_\gs: A\cap W_\gs\neq\0\},
\enq
and
the \emph{ordering} of $\{\xi_A:A\cap W_\gs\neq\0, \xi_A>\gs\}$.

Notice that $W_\gs$ and the set in \eqref{xiAs'} are enough to tell us
whether $\gL>\gs$, which by Lemma~\ref{stoppingLemma}
holds w.h.p.
If it does hold---which we now assume---then the above choices determine
$\{A\in \bG_\gL: A\cap W_\gs\neq\0\}$,
so in particular, for each $x\in W_\gs$, the first (under $\prec$)
$\gL$-edge, say $A_x$, containing $x$.
(They do not determine $\gL$, but we don't need this and
avoid conditioning on a zero-probability event.)

By Lemma~\ref{stoppingLemma}, w.h.p.
\beq{whp1}
\mbox{$|W_\gs|< n^{2\ga}$ and the $A_x$'s are distinct and disjoint}
\enq
(if $\gL<\gb$, as in (b) of the lemma, then the $A_x$'s are all in
$\bG_\gb$, so the second part of \eqref{whp1} is contained in (c));
so we assume these properties and set $U=\cup_{x\in W_\gs}A_x\sm W_\gs$.

\mn

Next, choose (and condition on)
\beq{Aings}
\{A\in \bG_\gs:A\cap U\neq \0=A\cap W_\gs\}
\enq
(from \eqref{xiAs'} we already know the members of $\bG_\gs$ that do meet
$W_\gs$).  Set
\beq{n'bd}
W=V\sm (W_\gs\cup U),  ~~~ n'=|W| > n-rn^{2\ga}
\enq
(using the first part of \eqref{whp1}),
and
\[
\bH^* =\bG_\gs[W]
\]
(meaning, as for graphs, the set of edges of $\bG_\gs$ contained in $W$).
Again by Lemma~\ref{stoppingLemma}, w.h.p.
\beq{whp2}
\mbox{no vertex of $W$ lies in more than one $\gs$-edge meeting $W_\gs\cup U$,}
\enq
and we add this assumption to those above.

\mn

Since $\Phi(\bH_T)=\Phi(\bG_\gL)\geq \Phi(\bH^*)$,
Theorem~\ref{ThmY'} will follow from
\beq{h*pm}
\mbox{w.h.p.
$~\Phi(\bH^*)> \left[e^{-(r-1)}\log n\right]^{n/r}e^{-o(n)}$.}
\enq
\nin
We will get this from Theorem~\ref{ThmZ}.

\mn

For $x\in W$ let
\beq{gd0gd}
\gd_x =\gd_0 -|\{A\in \bG_\gs:x\in A,A\cap (W_\gs\cup U)\neq\0\}|
\in \{\gd_0,\gd_0-1\}
\enq
(with the membership assertion given by \eqref{whp2}),
and notice that 
\beq{reduction1}
\mbox{\emph{$\bH^*$ is distributed as $\bH_{W,\gs}$ conditioned on
$\mL:=\{d(x)\geq \gd_x~\forall x\in W\}$}}
\enq
(where $\bH_{W,\gs}$---and $\bH_{W,m}$ below---have the obvious meanings
and 
$d$ is degree in $\bH_{W,\gs}$).

\mn

To complete the reduction to Theorem~\ref{ThmZ} we then just want to
replace $\bH_{W,\gs}$ by a suitable combination of $\bH_{W,m}$'s.
(Note \eqref{gd0gd} says the $\gd_x$'s are as in the theorem.)

\mn

By \eqref{prvWgs} and Harris' Inequality \cite{Harris}
we have
\beq{PrhW}
\pr(\bH_{W,\gs}\models \mL)> (1-n^{-1+\ga})^{n'} ~~~(\sim \exp[-n^\ga]),
\enq
with $\ga\sim \eps\log (e/\eps)$ (as in \eqref{prvWgs}, the substitution of $n'$ for $n$ 
and $\gd_x$ for $\gd_0$ having no significant effect).
On the other hand, with
\[\mu:=\E|\bH_{W,\gs}| =\Cc{n'}{r}\gs ~~ (\sim (n/r)\log n),
\]
$\gc=n^{-1/3}$ (say), and
$I = ((1-\gc)\mu,(1+\gc)\mu)$, Theorem~\ref{T2.1}
gives
\[
\pr(|\bH_{W,\gs}|\not\in I) < \exp[-\gO(n^{1/3})],
\]
which with \eqref{PrhW} implies
$\pr(|\bH_{W,\gs}|\not\in I|\mL)<
\exp[-\gO(n^{1/3})].$

According to Theorem~\ref{ThmZ} there is
\beq{nvsn'}
\Phi^* ~>~ [e^{-(r-1)}\log (n')]^{n'/r}e^{-o(n')}
~=~ [e^{-(r-1)}\log n]^{n/r}e^{-o(n)}
\enq
(the equality follows easily from \eqref{n'bd}) such that
\[
\max_{m\in I}\pr(\Phi(\bH_{W,m})< \Phi^*|\mL) =o(1).
\]
So, finally,
\begin{eqnarray*}
\mbox{$\pr(\Phi(\bH_{W,\gs})< \Phi^*|\mL)$}  &=&
\mbox{$\sum_m \pr(|\bH_{W,\gs}|=m|\mL) \pr(\Phi(\bH_{W,m})< \Phi^*|\mL)$} \\
&<&
\mbox{$\max_{m\in I}\pr(\Phi(\bH_{W,m})< \Phi^*|\mL)+o(1)$}\\
&=& o(1),
\end{eqnarray*}
which gives \eqref{h*pm}.\qed

\mn
\textbf{Acknowledgment.}
Thanks to Eyal Lubetzky for stimulating conversations.

\section{Appendix:  generics}\label{Reg}

Here we prove \eqref{Ri}.
We regard this item as a necessary (actually, for anyone who's gotten this
far surely
\emph{un}necessary) evil and aim to be brief.

\mn

With $D_t=D_{\bH_t}$,
\eqref{Ri} says
\beq{Hisat}
\mbox{w.h.p. $\bH_t$ satisfies \eqref{Rg1}-\eqref{Rg3} with $D=D_t$ for all $t\leq T$.}
\enq
For \eqref{Rg1}, \eqref{Rg3} and the upper bound in \eqref{Rg2}, a naive
union bound will suffice here, as
failure probabilities for individual $t$'s are very small.
A little more care is needed for the lower bound in \eqref{Rg2}, since
for $t$ near $T$ we can only say
\beq{HiRg3}
\pr(\mbox{$\bH_t$ violates \eqref{Rg2}}) < n^{-\ga}
\enq
with $\ga$ some small (positive) constant depending on $\eps$.
But even this is enough:
with $M_t=\C{n}{r}-t$ ($=|\bH_t|$)
and
\[I = \{t: \mbox{$M_t = 2^iM$ for some $i\in \{0,1,\ldots\}$}\}
\]
(recall $M=M_T$),
\eqref{Rg2} holds for all $t\leq T$ if it holds for all $t\in I$,
since then for $t'=\min\{l\in I:l\geq t\}$ we have
$D_t\leq 2D_{t'}$
and $\gd_{\bH_t}\geq \gd_{\bH_{t'}}=\gO(D_{t'})=\gO(D_t)$;
and \eqref{HiRg3} gives
$\sum_{t\in I}\pr(\mbox{$\bH_t$ violates \eqref{Rg2}}) =O(n^{-\ga}\log n)$.

\mn

We proceed to failure probabilities, now
writing $d_t$ for degree in $\bH_t$ and
beginning with
\eqref{Rg1}.
For $W\sub V$, we have
\[
\mbox{$\xi(W):=\sum_{v\in W}d_t(v) =\sum_{j=1}^rj\xi_j(W),$}
\]
where $\xi_j(W):=|\{A\in \bH_t: |A\cap W|=j\}|$ is hypergeometric with mean
\[
\mbox{$\frac{\C{|W|}{j}\C{n-|W|}{r-j}}{\C{n}{r}}M_t.$}
\]
If $\theta=\theta(n)=(\log n)^{-1/3}$ (say) and $|W|=\theta n$, then
$\E\xi_j(W)\sim \theta^j\C{r}{j}M_t$ and
\[
\mu:= \E\xi(W)\sim \E\xi_1(W)\sim \theta rM_t=\theta nD_t >\theta n\log n;
\]
so for $\gl =\theta \mu$,
Theorem~\ref{T2.1} gives $\pr(|\xi_j(W)-\E\xi_j(W)|> \gl)
< \exp[-\gO(\theta^2 \mu)]$ for $j\in [r]$
(with the true value much smaller if $j\neq 1$),
implying
\beq{prxiW}
\pr(|\xi(W)-\mu|> r\gl)< \exp[-\gO(\theta^2 \mu)]
= \exp[-\gO(\theta^3 n\log n)].
\enq
But if \eqref{Rg1} fails (for $\bH_t$) then there \emph{must} be some $W\in \C{V}{\theta n}$ with
$  
|\xi(W)-\mu|>r\gl,
$   
%
and \eqref{prxiW} bounds the probability that this happens by
\[
\Cc{n}{\theta n}e^{-\gO(\theta^3 n\log n)}
<\exp[\theta n\log (e/\theta)-\gO(\theta^3 n\log n)] = e^{-\gO(n)} ~(= n^{-\go(1)}).
\]
This gives \eqref{Rg1}.

\mn

For \eqref{Rg2} we apply Theorems~\ref{T2.1} and \ref{Cher'} to the
$d_t(v)$'s, each of which is hypergeometric with
mean $D_t>(1+\eps)\log n$.
For the upper bound, Theorem~\ref{Cher'} gives (say)
\[
\pr(d_t(v) > 3rD_t) < \exp[-3rD_t\log (3r/e)] < n^{-3r}
\]
so the probability that some $d_t(v)$ exceeds $3rD_t$ is less than $n^{-2r-1}$.
For the lower bound, with
$\gc= \eps/(2\log (1/\eps))$,
a simple calculation using
the first bound in \eqref{eq:ChernoffLower}
(\emph{cf.}\ \eqref{prvWgs};
the weaker second bound will not do here) gives (say)
\beq{absurd}
\pr(d_t(v) <\gc D_t) < n^{-(1+\eps/3)},
\enq
implying \eqref{HiRg3} (and, as discussed above, the lower bound in \eqref{Rg2}).

\mn

Finally, for \eqref{Rg3}:  Each codegree
$d_t(v,w)$ is hypergeometric with mean $(r-1)D_t/(n-1)$; so for $\vs$ with (say)
$1\gg \vs\gg \max\{D_t^{-1},n^{-1/2}\}$,
Theorem~\ref{Cher'} gives
\[
\pr(d_t(v,w)>\vs D_t) < \exp[-\vs D_t\log (e\vs n/r)] = n^{-\go(1)}.
\]

\end{document}